\documentclass[10 pt]{amsart}
\usepackage{latexsym, amsmath, amssymb, longtable, booktabs,amscd,microtype,booktabs,cases}
\usepackage[square]{natbib}
\usepackage{hyperref}

\bibpunct{[}{]}{,}{n}{}{;} 

\numberwithin{equation}{section}
\begin{document}
\newcommand\A{\mathbb{A}}
\newcommand\C{\mathbb{C}}
\newcommand\G{\mathbb{G}}
\newcommand\N{\mathbb{N}}
\newcommand\T{\mathbb{T}}
\newcommand\B{\mathcal{B}}
\newcommand\sO{\mathcal{O}}
\newcommand\sE{{\mathcal{E}}}
\newcommand\tE{{\mathbb{E}}}
\newcommand\sF{{\mathcal{F}}}
\newcommand\sG{{\mathcal{G}}}
\newcommand\GL{{\mathrm{GL}}}
\newcommand\Ham{\mathbb{H}}
\newcommand\HH{{\mathrm H}}
\newcommand\mM{{\mathrm M}}
\newcommand\fS{\mathfrak{S}}
\newcommand\fP{\mathfrak{P}}
\newcommand\fQ{\mathfrak{Q}}
\newcommand\Qbar{{\bar{\Q}}}
\newcommand\sQ{{\mathcal{Q}}}
\newcommand\sP{{\mathbb{P}}}
\newcommand{\Q}{\mathbb{Q}}
\newcommand{\tH}{\mathbb{H}}
\newcommand{\Z}{\mathbb{Z}}
\newcommand{\F}{\mathbb{F}}
\newcommand\gP{\mathfrak{P}}
\newcommand\Lim{\mathrm{Lim}}
\newcommand\Sup{\mathrm{Sup}}
\newcommand\Max{\mathrm{Max}}
\newcommand\Gal{{\mathrm {Gal}}}
\newcommand\SL{{\mathrm {SL}}}
\newcommand\Hom{{\mathrm {Hom}}}
\newcommand{\legendre}[2] {\left(\frac{#1}{#2}\right)}
\newcommand\iso{{\> \simeq \>}}
\newcommand{\f}{\matcal{f}}
\def\g{\mathfrak{g}}
\def\k{\mathfrak{k}}
\def\z{\mathfrak{z}}
\def\h{{\mathfrak h}}
\def\gl{\mathfrak{gl}}

\def\Ext{{\rm Ext}}
\def\Hom{{\rm Hom}}
\def\Ind{{\rm Ind}}
\def\Sym{{\rm Sym}}
\def\Coh{{\rm Coh}}

\def\GL{{\rm GL}}
\def\SO{{\rm SO}}
\def\O{{\rm O}}

\def\R{\mathbb{R}}
\def\C{\mathbb{C}}
\def\Z{\mathbb{Z}}
\def\Q{\mathbb{Q}}
\def\A{\mathbb{A}}

\def\w{\wedge}

\def\Cat{\mathcal{C}}
\def\HC{{\rm HC}}
\def\HCat{\Cat^\HC}
\def\proj{{\rm proj}}

\def\to{\rightarrow}
\def\To{\longrightarrow}

\def\1{1\!\!1}
\def\dim{{\rm dim}}

\def\th{^{\rm th}}
\def\isom{\approx}

\def\CE{\mathcal{C}\mathcal{E}}
\def\cE{\mathcal{E}}
\def\cO{\mathcal{O}}
\def\cP{\mathcal{p}}
\def\cI{\mathcal{I}}

\newtheorem{thm}{Theorem}
\newtheorem{theorem}[thm]{Theorem}
\newtheorem{cor}[thm]{Corollary}
\newtheorem{conj}[thm]{Conjecture}
\newtheorem{prop}[thm]{Proposition}
\newtheorem{lemma}[thm]{Lemma}
\theoremstyle{definition}
\newtheorem{definition}[thm]{Definition}
\newtheorem{remark}[thm]{Remark}
\newtheorem{example}[thm]{Example}
\newtheorem{claim}[thm]{Claim}

\newtheorem{lem}[thm]{Lemma}

\newtheorem{dfn}{Definition}

\theoremstyle{remark}
\newtheorem*{fact}{Fact}
% type user-defined commands here
\makeatletter
\def\imod#1{\allowbreak\mkern10mu({\operator@font mod}\,\,#1)}
\makeatother

\title{$p$-adic $L$-functions for $\GL_n$}
\author{\bf Debargha Banerjee \ \ \& \ \ A. Raghuram}
\address{Department of Mathematics\\ Indian Institute of Science Education and Research, Pashan, Pune Maharashtra 411021, India.} 
\email{debargha@iiserpune.ac.in}
\email{raghuram@iiserpune.ac.in}
\date{\today}
\begin{abstract} 
These are the expanded notes of a mini-course of four lectures by the same title given in the workshop ``$p$-adic aspects of modular forms" held at IISER Pune, in June, 2014. 
We give a brief introduction of  $p$-adic $L$-functions attached to certain types of  
automorphic forms on ${\rm GL}_n$ with the specific aim to understand 
the $p$-adic symmetric cube $L$-function attached to cusp forms on $\GL_2$ over rational numbers. 
\end{abstract}

\maketitle
 
\setcounter{tocdepth}{1}
\tableofcontents{}

The aim of this survey article is to bring together some known constructions of the $p$-adic $L$-functions associated to cohomological, cuspidal
automorphic representations on $\GL_n/\Q$. In particular, we wish to briefly recall the various approaches to construct $p$-adic $L$-functions 
with a focus on the construction of the $p$-adic $L$-functions for the $\Sym^3$ transfer of a cuspidal automorphic representation $\pi$ of $\GL_2/ \Q$. 
We note that $p$-adic $L$-functions for modular forms or automorphic representations are defined using $p$-adic measures. 
In almost all cases, these $p$-adic measures are constructed using the fact that the $L$-functions 
have integral representations, for example as suitable Mellin transforms. 
Candidates for distributions corresponding to automorphic forms can be written down using such integral representations of the $L$-functions at the critical points.  The well-known Prop.\,\ref{distributions} is often used to prove that they are indeed distributions, which is usually a consequence of the defining relations of the Hecke operators. Boundedness of these distributions are shown by proving certain finiteness or integrality properties, giving the sought after $p$-adic measures.  

\smallskip

In Sect.\,\ref{sec:p-adic-l-fn}, we discuss general notions concerning $p$-adic $L$-functions, including our working definition of 
what we mean by a $p$-adic $L$-function.  As a concrete example, we discuss the construction of the 
$p$-adic $L$-functions that interpolate critical values of $L$-functions attached to modular forms. Manin~\cite{Manin} and Mazur and Swinnerton-Dyer~\cite{Mazur} discovered how to construct those $p$-adic measures by defining 
a distribution such that  (1) it takes value in $\overline{\Q}$, and (2) these take value in a finite generated $\Z_p$-module. 
The last condition will ensure that these distributions are indeed  $p$-adic measures. 

\smallskip

In Sect.\,\ref{sec:sym-powers}, we discuss some basic facts about Langlands principle of functoriality, focusing 
mainly on the $\Sym^n$ transfer of an automorphic representation of $\GL_2/\Q$ giving an automorphic representation of $\GL_{n+1}/\Q$. We approach $L$-functions attached to $\Sym^3$ transfer of automorphic representations 
via instances of Langlands functoriality.  

\smallskip

In Sect.\,\ref{sec:lfnGL4}, we study the $p$-adic $L$-functions for cuspidal automorphic representation for $\GL_4/\Q$ that admit a so-called `Shalika model,' following the exposition of Ash and Ginzburg \cite{ash-ginzburg}.
(The reader is also referred to a forthcoming article by Dimitrov, Januszewski and the second author
\cite{dimitrov-januszewski-r}.) The symmetric cube transfer of a cuspidal representation $\pi$ of $\GL_2/\Q$ is a representation of $\GL_4/\Q$, whose standard degree four $L$-function is the symmetric cube $L$-function, and to which the results of \cite{ash-ginzburg} are applicable.

\smallskip

In Sect.\,\ref{sec:gl3-gl2}, we discuss $p$-adic $L$-functions for $\GL_3 \times \GL_2/ \Q.$  
We construct $p$-adic $L$-functions for the $\Sym^3$ transfer of a cuspidal representation $\pi$ of $\GL_2/\Q$ 
as a quotient of the $p$-adic $L$-function for $\GL_3 \times \GL_2$ applied to $\Sym^2(\pi) \times \pi$, and 
the $p$-adic $L$-function for $\GL_2$ attached to $\pi.$  This method produces the symmetric cube $p$-adic 
$L$-function in the quotient field of the Iwasawa algebra. We hope to get an element of the Iwasawa algebra, corresponding to the $\Sym^3$ transfer of automorphic representations $\pi$ of $\GL_2/ \Q;$ see the discussion in 
Sect.\,\ref{sec:sym-3-II}. We end the introduction by pointing to a tantalizing possibility that one can get 
$p$-adic $L$-functions for symmetric cube transfers using the integral representations of symmetric cube 
$L$-functions in Bump, Ginzburg and Hoffstein~\cite{Bump}.

\bigskip

\section{What is a $p$-adic $L$-function?}
\label{sec:p-adic-l-fn}
We follow the exposition in \cite{PollackDuke} to define $p$-adic $L$-functions. Fix an odd prime $p$ and an embedding 
$i_p:\overline{\Q} \rightarrow \C_p=\widehat{\overline{\Q}}_p$. The field $\C_p$ is called the Tate field. 
Fix a valuation $v_p$ on the Tate field and let $O_p$ be the ring of integers of $\C_p$.
We also fix an embedding  $i_{\infty}: \overline{\Q} \rightarrow \C$.

\medskip
\subsection{\bf The weight space $X_p$}
Let $X_p$ be the set of continuous homomorphisms $ \Z_p^{\times} \rightarrow \C_p^{\times}$, i.\,e.\,, 
\[
X_p=\Hom_{\rm{Cont}}(\Z_p^{\times}, \C_p^{\times}).
\]
We call $X_p$ the weight space. The elements of $X_p$ are called $p$-adic characters. Recall, 
we have $\Z_p^{\times}=(\Z/p\Z)^{\times} \times (1+p \Z_p)$. 
A character is said to be tame if it is trivial on $1+p \Z_p$ and it is called wild if the character is trivial on $(\Z/p\Z)^{\times}$. 
Every character can be uniquely written as $\chi=\chi_t \cdot \chi_w$ with $\chi_t$ tame and $\chi_w$ wild.

\begin{lemma}
\label{opendisc}
The weight space $X_p$ can be identified with a disjoint union of $p-1$ copies of 
the open unit disc $\B := \{u \in \C_p \ | \ |u-1|_p<1 \}$ of $\C_p.$
\end{lemma}

\begin{proof}
Fix a topological generator $\gamma$ of $1+p\Z_p$. For $u \in \C_p^{\times}$ with 
$|u-1|<1$ define a particular wild character $\chi_u \in X_p$ as 
\begin{equation}
\label{tamewild}
\Z_p^{\times} \rightarrow 1+p\Z_p \rightarrow \C_p^{\times},
\end{equation}
where the first map sends the topological generator $\gamma$ to $1+p$ and the second map sends\,$1+p$ 
to $u$. The set $\{\chi_u| \, u \in \C_p, \, |u-1|_p<1 \}$ is the set of all wild characters,
since the continuity of a character $\chi$ requires that $|\chi(\gamma)-1|_p<1$. 
Let $\psi$ be a tame character on $\Z_p^{\times}$ . The mapping $u \rightarrow \psi \chi_u$
identifies the open unit disc of $\C_p$ with the set of characters on $\Z_p^{\times}$
with tame part equal to $\psi$. Since there are only $p-1$ distinct tame characters on $\Z_p^{\times}$, we have that
$X_p$ is a union of as many copies of $\B.$
\end{proof}

We list some properties of $X_p$ which are relevant for this article. 
\begin{itemize}
 \item 
 The set $X_p$ is a group under pointwise multiplication. 
 \item 
 The torsion subgroup of $X_p$ is exactly the set of all Dirichlet characters of $p$-power conductor.
 \item 
 $X_p$ has the structure of a $p$-adic Lie group. 
 \item 
 $X_p$ contains the $p$ components of all adelic characters $\chi=\prod_{l \leq \infty} \chi_l: \Q^{\times} \backslash \A^{\times} \rightarrow \C^{\times}$, 
 which are of $p$-power conductor and of finite order ($\chi_{\infty}=\1$ or $\chi_{\infty}=$ sgn). 
\end{itemize}
We give some examples of $p$-adic characters.
\begin{enumerate}
\item
The characters of the form 
\[
x^j \chi(x)
\]
where $j$ is an integer, and $\chi(x)$ is a character of finite order.

\item 
For $x \in \Z_p^{\times}$, we write $x=\omega(x) <x>$ with $\omega(x)$ a $(p-1)$ root of unity and $<x>$ lies in $1+p \Z_p$. For $s \in \Z_p$, we 
define 
\[
\chi_s(x)=<x>^s=\sum_{r=0}^{\infty} \frac{s^r}{r!} (\log<x>)^r.
\]
\end{enumerate}

A $p$-adic $L$-function is a $p$-adic analytic function 
$L_p:X_p \rightarrow \C_p$ that 
interpolates the {\it algebraic parts of the complex values} of some $L$-function associated to an automorphic representation (see \S\,\ref{automorphic}) or a motive (see \S\,\ref{Motive}).
The $p$-adic $L$-function attached to an automorphic representation $\pi$ will be denoted by $L_{p, \pi}$ and the $p$-adic $L$-function attached to a motive $M$ will be  denoted by $L_{p,M}$.
For a Dirichlet character $\psi$, the value of $L_{p, \pi}$ at the special elements of $X_p$ of the form $\chi_{k,\psi}:x_p \rightarrow \psi(x) x_p^k$ coincides with the algebraic parts of the special $L$-values
of $\pi \otimes \psi$ at the integer $k$. 
A $p$-adic function is analytic if it is given by power series with $p$-adic coefficients on copies of the 
unit disc of $\C_p$.

\medskip
\subsection{\bf \texorpdfstring{$p$}{X}-adic measures}
\label{measures}
We will now define  $p$-adic distributions and $p$-adic measures. Let $X$ be a compact, open subset of $\Q_p$ such as $\Z_p$ or $\Z_p^{\times}$. 
A $p$-adic distribution $\mu$ on $X$ 
is a continuous linear map from the $\C_p$ vector space $C^\infty(X, \C_p)$ of locally constant functions on 
$X$ to $\C_p$, which we write as: 
$$
\mu \ \in \ {\rm Hom}_{\C_p}(C^\infty(X, \C_p), \, \C_p).
$$
If $f$ is a locally constant function then $\mu(f)$ is also denoted $\int_X f d\mu$. Equivalently, a $p$-adic distribution $\mu$ on $X$ is an additive 
map from the set of compact, open subsets of $X$ to $\C_p$.  
The following proposition (see, for example, Koblitz~\cite[II.3]{Koblitz}) is 
very effective in constructing distributions. 

\begin{prop}
\label{distributions}
An interval is a set of the form $a+p^n \Z_p.$ A map $\mu$ from the set of intervals of $X$ to $\Q_p$, which satisfies the equality 
\[
\mu(a+p^n \Z_p) \ = \ \sum_{b=0}^{p-1} \mu(a+b p^n +p^{n+1} \Z_p)
\]
for $a+p^n \Z_p \subset X$, extends uniquely to a $p$-adic distribution on $X$.  
\end{prop}

Following Vishik~\cite{Vishik} and Amice-Velu~\cite{AmiceVelu}, we define $h$-admissible 
measures.

\begin{definition}[$h$-admissible measure]
Let $C^h(\Z_p^{\times})$ be the space of functions $f:\Z_p^{\times} \rightarrow \C_p$ which are locally given by polynomials of degree at most
$h$. Let $C^{la}$ be the $\Z_p$-module of all locally analytic functions and  $C(\Z_p^{\times})$ be the space of all continuous functions. 
We have  inclusions:  
\[
 C^1(\Z_p^{\times}) \subset \cdots  \subset C^h (\Z_p^{\times}) \subset \cdots \subset C^{la}(\Z_p^{\times}) \subset  C(\Z_p^{\times}). 
\]
Let $\chi_X$ be the characteristic function of the set $X$. 
An $h$-admissible measure $\mu$ on $\Z_p^{\times}$ is a continuous linear map $\mu:C^h(\Z_p^{\times}) \rightarrow \C_p$ such that 
\[
|\mu((x-a)^i \chi_{a+p^n \Z_p})|=O(p^{e(h-i)})
\]
for $0 \leq i \leq h$ and $e$ tends to infinity. 
\end{definition}

\begin{thm}[\cite{Vishik}]
 An $h$-admissible measure $\mu$ extends to a linear map on the space of all locally analytic
functions on $\Z_p^{\times}$.
\end{thm}

Let $K$ be a finite extension of $\Q_p$ and let $\mu$ be a 
$K$-valued measure on $\Z_p^{\times}$. We wish to understand how we can integrate functions with 
respect to this measure. Let $R_m$ be a system of representatives from $(\Z_p \slash p^m \Z_p)^{\times}$ in $\Z_p^{\times}$,
and let $f : \Z_p^{\times} \rightarrow K$ be a function. Consider the
``Riemann sum"
\[
S(f;R_m)=\sum_{b \in R_m} f(b) \mu(b+p^m \Z_p).
\]

The following fundamental theorem is due to Manin \cite{Manin}.
\begin{thm}
\label{Manin}
There exists a unique limit
\[
\lim \, S(f, R_m):=\int_{\Z_p^{\times}} f d \mu,
\]
taken over all $R_m$ as $m$ tends to $\infty$, provided that the following conditions are satisfied
\begin{itemize}
 \item 
 The measure $\mu$ is of moderate growth; that is, by definition,
\[
\epsilon_m=\mathrm{{Max}_b} \, |\mu(b+p^m \Z_p)| p^{-m} \rightarrow 0
\]
as $m \rightarrow \infty$.
\item 
 The function $f$ satisfies ``Lipschitz condition'', i.e., there exists a constant $C$
such that if $b \equiv b' \pmod {p^m}$ then 
\[
 |f(b)-f(b')| < Cp^{-m},
\]
as $m \rightarrow \infty$.
\end{itemize}
\end{thm}

We note that the set of locally constant functions on $\Z_p$ are dense in the set of continuous functions on 
$\Z_p$. 
A $p$-adic distribution is called a $p$-adic measure if it is bounded, i.e., if there is a real number $N$ such that 
$|\mu(U)| \leq N$ for all compact, open subsets $U$ of $X$.

\medskip
\subsection{\bf \texorpdfstring{$p$}{X}-adic $L$ functions}
Kubota and Leopoldt first constructed 
$p$-adic meromorphic functions that interpolate {\it special values} of Riemann zeta function and more generally special values of Dirichlet $L$-functions. 
The existence of these meromorphic functions is equivalent to congruences of (generalized) Bernoulli numbers. 
An integer $k$ can be viewed as a character $x_p^k : x \rightarrow  x^k$. The
construction of Kubota and Leopoldt is equivalent to the existence of a $p$-adic
analytic function $\zeta_p : X_p \rightarrow  \C_p$ with a single pole at the point $x=x_p^{-1},$  which
$p$-adically becomes a bounded holomorphic function (given by power series) on $X_p$ after multiplication by the
elementary factor $(x_p x - 1) (x \in  X_p )$, and is uniquely determined by the
interpolation property 
\[
 \zeta_p(x_p^k)=(1-p^k) \zeta(-k). 
\]
The $p$-adic $\zeta$-function is constructed by defining a $p$-adic measure on $\Z_p^{\times}$ 
with values in $\Z_p$ such that 
\[
\int_{\Z_p^{\times}} x_p^k d \mu=(1-p^k)\zeta(-k).
\]
(See, for example, Koblitz~\cite[II.6]{Koblitz}.)

\begin{definition}[$p$-adic $L$-functions]
A $p$-adic measure $\mu$ on $\Z_p^\times$ gives a $p$-adic $L$-function $L_{p,\mu} : X_p \to \C_p$ 
whose value on a character $\chi \in X_p$ is given by: 
\[
L_{p,\mu}(\chi)=\int_{\Z_p^{\times}}\chi d \mu.
\]
\end{definition}

%The $p$-adic $L$-functions can be obtained by evaluating characters on this measures.  For a $p$-adic measure $\mu$ and a $p$-adic character $\chi$, the value of the $p$-adic $L$-function at $\chi$ is 

\medskip
\subsection{\bf \texorpdfstring{$p$}{X}-adic measures on \texorpdfstring{$\Z_p^{\times}$}{X} and power series}
\label{powerseries}
The following theorem of Manin~\cite{Manin} gives an explicit connection between bounded measures and elements of the Iwasawa algebra. 

\begin{thm}
 Let $\mu$ be a $K$-valued measure on $\Z_p^\times$ of moderate growth (see Thm.\,\ref{Manin}). 
 For each tame character $\chi_{\tt t}$ of $\Z_p^\times$ there is a unique power series 
 $g_{\mu, \chi_{\tt t}}  \in K[[T]]$ that is convergent for any specialization of $T \in p\Z_p$, such that 
 for all $\chi \in X_p$ we have 
 \[
L_{p,\mu}(\chi) \ = \ g_{\mu, \chi_0}(\chi_1(1+p)-1). 
\]
where $\chi_0$ (resp., $\chi_1$) is the tame (resp., wild) component of $\chi.$ 
\end{thm}

It is easy to see that $\chi_1(1+p)-1$ lies in $p\Z_p$ and so the right hand side is convergent.

\medskip
\subsection{\bf Relations between $h$-admissible measures and power series with bounded growth}
Following~\cite{Vishik} and~\cite{AmiceVelu}, we recall the relation between $h$-admissible $p$-adic measures 
and $p$-adic power series of bounded growth. 
Recall, the open disc $\B = \{u \in \C_p \ | \ |u-1|_p<1 \}.$  Suppose $f$ is an analytic function 
on $\B$ with the Taylor series expansion around $1$ given by 
$f(X)=\sum_{n \geq 0} b_n (X-1)^n$. 

\begin{definition}
[Modulus function]
We define the modulus function of $f$ to be 
$$
M_f(r) \ = \ \Sup_{|x-1|=r}|f(x)| \ = \ 
\Max_n \, |b_n r^n|.
$$
\end{definition}

\begin{definition}
[Big $O$ and small $o$ for $p$-adic analytic functions]
Suppose $f$ and $g$ be two $p$-adic analytic functions on $\B$, we say that 
\begin{enumerate}
\item $f=O(g)$ if 
$\lim_{r \rightarrow 1^{-}} \frac{M_f(r)}{M_g(r)}$ is finite as $r \rightarrow 1$, and 
\item $f=o(g)$ if they satisfy the stronger condition 
$\lim _{r \rightarrow 1^{-}}\frac{M_f(r)}{M_g(r)} = 0$. 
\end{enumerate}
\end{definition}

For example, if $g(X)=\log_p(X)^k$ and $f(X)=\sum_{n \geq 0} b_n (X-1)^n$ then $f=o(g)$ 
if and only if $|b_n|=o(n^k)$. 
For function $f$ and $g$ analytic on $X_p$, we say $f=O(g)$ or $f=o(g)$ if on each of the component 
isomorphic to $\B$, the functions $f$ and $g$ have the property.

\medskip
\subsection{\bf \texorpdfstring{$p$}{X}-adic measures on \texorpdfstring{$\Z_p$}{X} and power series}
In this section, we explore the connection between 
$p$-adic measures on $\Z_p$ and various power series ring \cite{Lang}. 
Measures on $\Z_p$ give rise to measures on $\Z_p^{\times}$ by restriction. On the other hand, measures 
on $\Z_p^{\times}$ produce measures on $\Z_p$ by first restricting to $1+p\Z_p$ and then via 
the identification of $1+p\Z_p$ with $\Z_p$.

Recall, a measure $\mu$ on $\Z_p$ is a bounded linear functional on the $\C_p$-vector space 
$\C(\Z_p, \C_p)$ of all continuous $\C_p$-valued functions on $\Z_p,$ i.e., 
there exists a constant $B > 0$ satisfying $|\mu(f)| < B |f|$ for all $f \in  C(\Z_p,\C_p)$.
The smallest possible $B$ is called the norm of the measure $\mu$ and is denoted $||\mu||_p$. 
With this norm, the set $M(\Z_p,\C_p)$ of measures on $\Z_p$ becomes a $\C_p$-Banach space.

Let $\C_p\{\{T\}\}$ be the $\C_p$-algebra of power series whose coefficients are in $\C_p$ and are bounded with respect to $v_p$. Define the norm  $\C_p\{\{T\}\}$ as the maximum of the absolute values of the coefficients. This is also a 
$\C_p$-Banach space. The {\it Amice transform} gives an isometry between these two Banach spaces, which we now proceed to describe.

\begin{definition}
\label{Amice}
(Amice Transforms)
The Amice transform of $\mu \in M(\Z_p,\C_p)$ is the power series 
\[
A_{\mu}(T) \ := \ \sum_{n=0}^{\infty} (\int_{\Z_p} {x \choose n} d\mu(x)) \, T^n \ = \ \int_{\Z_p} (1+T)^x d \mu(x). 
\]
\end{definition}
In the other direction, given a power series $F  = \sum_{n \geq  0} F_n T^n \in  \C_p\{\{T\}\}$ 
%with coefficients $F_n$ ($n \geq 0$), 
define $\mu_F$ on the `binomial coefficient functions' via: 
 \[
\int_{\Z_p} {x \choose n} d\mu_F \ = \ F_n. 
\]
Using a well-know theorem due to Mahler, one can show that this uniquely determines the measure 
$\mu_F$.

\begin{prop}
The map $\mu \rightarrow A_{\mu}$ is an isometry from $M(\Z_p,\C_p)$ to $\C_p\{\{T\}\}$.
\end{prop}

Since $A_\mu$ has bounded coefficients, for any specialization of $T = z$ with $v_p(z) >0$, the series 
$A_\mu(z)$ will converge. From the above definition, we have: 
\begin{lemma}
\label{inte}
If $v_p(z) >0$, then $$\int_{\Z_p} (1+z)^x d \mu(x)=A_{\mu}(z).$$ 
\end{lemma}

\smallskip

We briefly review power series with integral coefficients. For a finite extension $K$ of $\Q_p$,  define 
\[
A(K)= \{f \in K[[T]] \ |\ \mbox{$f(z)$ is convergent for any $z \in \C_p$ with $v_p(z) >0$} \}.
\]
The power series with coefficients in $O_K$ can be characterized in terms of their zeros (see \cite{PollackDuke}): 

\begin{lemma}
\label{finitezeros}
Let $K$ be a finite extension of $\Q_p$ . Then $f(T) \in A(K)$ has finitely many zeros if and
only if $f(T) \in O_K[[T]] \otimes K$.
\end{lemma}

\medskip
\subsection{\bf Convolution of two measures}
\label{Convolution} 
Let $\lambda$ and $\mu$ be two measures on $\Z_p$ with values in $K$, their convolution $\lambda * \mu$ is defined to be the measure 
\[
\int f d (\lambda * \mu)=\int \int f(x+y) d \lambda (x) d \mu (y). 
\]
Since $f$ is uniformly continuous, so $f \rightarrow \int_{\Z_p} f(x+y)d\mu(x)$ is continuous. 

\begin{lemma}
The multiplication of power series correspond to the convolution of measures on the additive group $\Z_p$, i.e., 
 $A_{\lambda *\mu}=A_{\lambda} A_{\mu}.$
 \end{lemma}
 
\begin{proof}
Consider the function $f(x)=z^x$ and $v_p(z-1)>0$.  By Lemma~\ref{inte}, we have 
\[
A_{\lambda * \mu}(z) \ =\ 
\int_{\Z_p} z^{x} (\lambda* \mu)(x) \ = \int_{\Z_p} z^{x+y} \lambda(x) \mu(y) \ = \ 
\int_{\Z_p} z^x \lambda(x)  \int_{\Z_p} z^y  \mu(y) \ = \ A_{\lambda}(z) A _{\mu}(z).
\]
\end{proof}

\medskip
\subsection{\bf $p$-adic $L$-functions for modular forms}
\label{Modularforms}
For modular forms, $p$-adic $L$-functions were constructed by Manin \cite{Manin}, Mazur and 
Swinnerton-Dyer \cite{Mazur} for ordinary primes using modular symbols. 
The construction has been extended for non-ordinary primes by Vishik~\cite{Vishik}, Amice-V\'elu~\cite{AmiceVelu} and Pollack~\cite{PollackDuke}. 
There are two known methods of construction of $p$-adic $L$-functions for modular forms at the ordinary primes:
(1)  Modular symbols, (2) Kato's  Euler systems. 
In the first method, $p$-adic measures are defined using the properties of these modular symbols and then the 
$p$-adic $L$-functions are given by 
integrating the characters of $p$-power conductors with respect to these $p$-adic measures. 
Several known constructions of 
the $p$-adic $L$-functions for automorphic forms use and generalize this method, as we will see later. 
We will not be dealing with Euler systems in this article. 

\smallskip

Let $f \in S_k(N,\epsilon)$ 
be a normalized holomorphic cusp form for $\Gamma_0(N)$ of weight $k \geq 2$ and character 
$\epsilon;$ assume that $f$ is a Hecke eigenform. Let $K(f)$ be the finite extension of $\Q$ 
generated by the Fourier coefficients of the modular form $f$ and let $\mathcal{O}(f)$ be the ring of integers 
of $K(f)$. Let $\alpha$ and $\beta$ be roots of the Hecke polynomial 
at $p$, i.e., 
\begin{equation}
\label{Hecke}
X^2-a_p X +\epsilon(p)p^{k-1}=(X-\alpha)(X-\beta).
\end{equation}
If $v_p(\alpha)=0$ ($p$ is ordinary for $f$)  then the $p$-adic $L$-function 
is a power series with coefficients in $\Z_p$. By the $p$-adic  Weierstrass preparation theorem, there are only finitely 
many zeros of this power series. If $v_p(\alpha)>0$, the $p$-adic $L$-function is not a 
power series with coefficients in $\Z_p$, and hence it has infinitely many zeros (see Lem.\,\ref{finitezeros}). 
If $0 < v_p(\alpha) < k-1$, Vishik and Amice-Velu studied $p$-adic $L$-functions for modular forms of weight greater than $2$; these are  power series (may 
not be with bounded coefficients) of bounded $p$-adic growth like $\log_p(T)$. 
If $a_p=0$ ($p$ is a supersingular prime for $f$), then it is not possible to apply 
the method of Vishik and Amice-V\'elu. 
In this case, Pollack discovered a method to remove certain special zeros of this power series
and constructed $p$-adic $L$-functions with co-efficients in $\Z_p$. 
The $p$-adic analytic function $L_{p,f}$ of bounded growth on $X_p$ is
exactly the Mellin transforms of $p$-adic $h$-admissible measures $\mu_f$ on $\Z_p^{\times}$: 
$L_{p,f}(\xi) \ = \ 
\int_{\Z_p^{\times}} \xi d \mu_f.$
In particular, $p$-adic $L$-functions are obtained  by integrating $p$-adic characters against $p$-adic $h$-admissible measures. We now describe the admissible measure $\mu_f$ corresponding to $f \in S_k(N,\epsilon)$ as above. 

\smallskip

For $f$ as above and a polynomial $P$ of degree less than $k-1$, define 
\[
\phi(f, P,r) \ = \ 2 \pi i \int_{i \infty}^r f(z) P(z) dz. 
\]
Let $L_f$ be the $\Z$-module generated by all $\phi(f, P, r)$ for all $r \in \Q$; then $L_f$ is finitely generated over $\Z$. 
We call a root $\alpha$ of $X^2-a_p X+\epsilon(p) p^{k-1}=0$ allowable if $ord_p(\alpha)<k-1$. 
For $$\eta(f,P, a, m) \ := \ \phi(f, P(mz-a),-\frac{a}{m}), $$ we define the plus and minus parts 
of $\eta$ by 
\[
\eta^{\pm}(f,P, a, m) \ = \ \frac{\eta(f,P, a, m) \pm \eta(f,P, -a, m)}{2}. 
\]
By a well known theorem of Manin, there exist ${\Omega_f^{\pm}}  \in \C^{\times}$ 
such that $\frac{\eta^{\pm}(f,P, a, m)}{\Omega_f^{\pm}} \in \mathcal{O}(f)$. 
We now define the period integral of $f$ by 
$\lambda^{\pm}(f,P, a, m)=\frac{\eta^{\pm}(f,P, a, m)}{\Omega_f^{\pm}} \in \mathcal{O}(f)$.
An admissible measure on  $\Z_p^{\times}$ associated to $f$ and $\alpha$  
is defined by the formula 
\begin{equation}
\label{explicitmeasure}
\mu_{f, \alpha}( P, a+p^n \Z_p) \ = \  \frac{1}{\alpha^n} \lambda^{\pm}(f,P,a, p^n)-\frac{\epsilon(p)p^{k-2}}{\alpha^{n+1}} \lambda^{\pm}(f,P, a, p^{n-1}).
\end{equation}
The $p$-adic $L$-function $L_{p, f, \alpha}$ is obtained by evaluating the 
characters of $p$-power orders on $\mu_{f, \alpha}$. The following is the main result on $p$-adic $L$-functions for modular forms:

\begin{theorem}
\label{Mazur}
(Manin, Mazur--Swinnerton-Dyer, Mazur--Tate--Teitelbaum, Vishik, Amice-V\'elu) Let $f$ be a
cuspidal normalized eigenform of weight $k $, level $N$ and character $\epsilon$.
Assume that $N$ is prime to $p$. Let $\alpha, \beta$  be the two roots of $X^2-a_p X + p^{k-1} \epsilon(p)=0$, and
choose one, say $\alpha$, with $v_p(\alpha) < k - 1$. 
There exists a unique function $L_{p, f, \alpha}:X_p \rightarrow \C_p $
that satisfies the following properties:
\begin{itemize}
 \item 
 (interpolation) For any character $\chi : \Z_p^{\times} \rightarrow \C_p^{\times}$ of finite image and of conductor $p^n$, and any
integer $j$ such that $0 \leq j \leq  k-2$, we have 
\[
 L_{p, f, \alpha}(x_p^j \chi) \ = \ e_{p, f, \alpha}(\chi, j) \, 
\frac{p^{n(j+1)} j!}{\Omega_f^{\pm}  G(\chi^{-1})\alpha^ n (-2\pi i)^j} \, 
L(f, {\chi}^{-1} , j + 1), 
\]
where $G(\chi^{-1})$ is the Gauss sum of $\chi^{-1}$ and 
$e_{p, f, \alpha}(\chi,j)=(1-\frac{\overline{\chi(p)} \epsilon(p)p^{k-2-j}}{\alpha})(1-\frac{\chi(p)p^j}{\alpha}).$
\smallskip
\item
(growth rate) The order of growth of $L_{p, f, \alpha}$ is $\leq v_p(\alpha)$.
\end{itemize}
\end{theorem}

In the $p$-ordinary case ($ord_p(a_p)=0$), there is a unique allowable $\alpha$. In this 
case, the corresponding distribution is a measure. We note that the measure grows at a faster rate if 
${\rm ord}_p(\alpha)>0$, 
since $\alpha$ is in the denominator of (\ref{explicitmeasure}).

\medskip
\subsection{\bf $p$-adic $L$-functions for automorphic forms}
\label{automorphic}
The above mentioned constructions of $p$-adic $L$-functions and $p$-adic measures that interpolates critical 
values of $L$-functions attached to modular forms can be generalised to get $p$-adic $L$-functions for 
cohomological, cuspidal, automorphic representations $\pi$.
For automorphic representations on $\GL_2$ over totally real 
number fields, the above $p$-adic $L$-functions were constructed by Manin~\cite{ManinJacquet}. This has been generalized by Haran~\cite{Haran}  for any number field. 
Mahnkopf ~\cite{mahnkopfpadic} and his student Geroldinger \cite{Mahnkopfstu} generalized this work for 
$\GL_3$. For $\GL_3 \times \GL_2 /\Q$, such $p$-adic 
$L$-functions were constructed by C.-G. Schmidt~\cite{schmidt-inv2}. 
His construction was 
generalised by Kazdhan, Mazur and Schmidt \cite{kms} to $\GL_n \times \GL_{n-1}/\Q,$ and 
Januszewski~\cite{janus-imrn} \cite{janus-preprint} for $\GL_n \times \GL_{n-1}$ over general number fields.   
In a different direction, one may say that Manin's construction was partially generalized to 
$\GL_{2n}$ by Ash--Ginzburg \cite{ash-ginzburg}. 
The general recipe involves studying the algebraic parts of the special 
values of complex $L$-functions 
for automorphic representations, and then using them to construct $p$-adic measures and hence 
$p$-adic $L$-functions. The existence of a $p$-adic measure will depend on several choices:
\begin{itemize}
  \item 
appropriate periods to make $L$-values algebraic;
 \item 
a root of the Hecke polynomial at $p$;
 \item 
critical points of the complex $L$-function associated to the automorphic form.
 \end{itemize}
We will elaborate on this recipe in a couple of situations in \S\,\ref{sec:lfnGL4} and \S\,\ref{sec:gl3-gl2} below.

\medskip
\subsection{\bf $p$-adic $L$-function for motives}
\label{Motive}
Following  Coates~\cite{Coates}, Panchishkin \cite{Panchiskin} and Dabrowski~\cite{Dabrowski}, we briefly discuss 
a general conjecture on the existence of $p$-adic $L$-function attached 
to a motive. Let $M$ be a pure motive over $\Q$ with coefficients in $\Q$ of weight $w= w(M)$ and rank $d = d(M)$. 
This motive has Betti, de Rham and $l$-adic realizations (for each prime $l$)
with cohomology groups $H_B(M), H_{DR}(M)$ and $H_l(M)$ which are  vector spaces
over $\Q$, $\Q$ and $\Q_l$, respectively, all of dimension $d.$ These groups are endowed with additional structures and comparison isomorphisms.
In particular, $H_B(M)$ admits an involution $\rho_B$ and there is a Hodge decomposition into $\C$-vector spaces
\[
\HH_B(M) \otimes \C=\bigoplus_{p+q=w} \HH^{p,q}(M).
\]
 Let $\rho_B$ acts on $\HH_B(M) \otimes \C$ via it's action on the first factor.
We have $\rho_B( \HH^{ i , j} ( M )) = \HH^{ j ,i} ( M )$.
Let $h(i , j ) = \mathrm{dim} \, \HH^{i,j}(M)$ which are called the Hodge numbers of $M$, 
and let $\mathrm{d}^{\pm} =\mathrm{d}^{\pm}(M)$ be the $\Q$-dimension of the $\pm$-eigenspace of 
$\rho_B$. Furthermore, $\HH_l(M)$ is a $\Gal(\overline{\Q}/\Q)$ module and we denote 
the corresponding representation by $\rho_l$.

\begin{definition}
[{\bf Hodge polygon};  see Panchiskin~\cite{Panchiskin}]
The Hodge polygon $P_H(M)$ is a continuous function on $[0,d]$, whose graph is a polygon 
joining the points
$$
\left(0,0\right), \ \cdots, \  \left(\sum_{i' \leq i} h(i',j), \sum_{i' \leq i } i'h(i',j)\right), \ \cdots, 
 \left(\sum_{i' \leq d} h(i',j), \sum_{i' \leq d } i'h(i',j)\right).
 $$  
Note that by purity, $j = w-i'.$
\end{definition}

\begin{definition}
[{\bf Newton Polygon of a polynomial}]
Let 
$$
P(T) \ = \ 1+a_1T+a_2 T^2+ \dots + a_dT^d \ = \ \prod_{i=1}^d (1-\alpha_iT)
$$ 
be a polynomial with coefficients 
in $\C_p$ and let the roots $\alpha_i$ of this polynomial be ordered such that $v_p(\alpha_i) \leq v_p(\alpha_{i+1})$ for all $i$. The {\it Newton polygon} of $P$ with respect to the $p$-adic 
valuation $v_p$ is defined to be the graph of the continuous, piecewise linear, convex 
function $f$ on $[0,d]$  obtained by joining the points 
$$
(0,0), \ \dots , \ (i,v_p(\alpha_i)), \ \dots , \ (d, v_p(\alpha_d)).
$$
\end{definition}

Let $I_p$ be the inertia subgroup of the decomposition group $\Gal(\overline{\Q}_p \slash \Q_p)$.
The $L$-function of the motive $M$ is defined as an Euler product 
$$
L(s,M)=\prod_p L_p(s,M),
$$ 
with the Euler factor at $p$ given by $L_p(s,M) = Z_p(X, M)^{-1}|_{X = p^{-s}}$ where the Hecke polynomial $Z_p(X,M)$ is defined as: 
\begin{equation}
\label{Heckemotive}
Z_p(X,M) \ := \ \mathrm{det}(1-\rho_l(\mathrm{Frob}^{-1}_p)X|{\HH_l(M)}^{I_p}) \ = \ 
\sum_{i=0}^d A_i(p)X^i \ = \ \prod_{i=1}^d(1-\alpha_i X).
\end{equation}
This is a polynomial with coefficients in $\Q_\ell$, and via the usual expectation of $\ell$-independence, the 
coefficients $A_i(p)$ are in $\Q,$ and so the roots $\alpha_i$ are in $\overline{\Q}$, but are thought of as elements of 
$\C_p$ via the fixed embedding $i_p: \overline{\Q} \rightarrow \C_p.$

\begin{definition}
[{\bf Nearly $p$-ordinary}, see Hida~\cite{HidaPoincare}]
We call a motive 
$M$ to be nearly $p$-ordinary if  
$$
P_{N,p}(M) = P_H(M).
$$
\end{definition}

The following is part of a conjecture on the existence of $p$-adic $L$-functions attached to pure motives. 
For the $\alpha_i$ as in (\ref{Heckemotive}), and for 
a Dirichlet character $\chi$ of conductor $c(\chi)$ and an integer $m$, define a factor at $p$ by 
\begin{equation*}
A_p(m, M(\chi)) \ = \ 
\begin{cases}
\prod_{i=d^++1}^d(1-\chi(p)\alpha_i p^{-m}) \prod_{i=1}^{i=d^+}(1-\chi^{-1}(p)\alpha_i^{-1} p^{m-1}), 
 & \text{if $p \mid c(\chi)$}\\
 (\frac{p^m}{\alpha_p^{(i)}})^{ord_p(c(\chi))}, & \text{if $p \nmid c(\chi)$}.
\end{cases}
\end{equation*}
For a pure motive $M$, let $\Lambda(s,M(\chi))$ be the completed $L$-function associated 
to the motive $M(\chi)=M \otimes \chi$.

\begin{conj}[\cite{Coates}, \cite{Dabrowski}, \cite{Panchiskin}] 
For any sign $\epsilon = \pm$, there exists a period $\Omega(\epsilon, M)$, and there exists a 
meromorphic function $L^{\epsilon}_{p,M} : X_p \rightarrow \C_p$, satisfying the following properties: 
\begin{itemize}
\item
For all but finite number of pairs $(m, \chi) \in X_p^{tor}$ such that $M(\chi)(m)$ is critical and 
$\epsilon_0 = ((-1)^m \epsilon(\chi))$, we have 
\[
L^{\epsilon_0}_{p,M}(\chi x_0^m)=G(\chi)^{-d\epsilon_0(M)} A_p(M(\chi),m) \frac{\Lambda(M(\chi),m)}{\Omega(\epsilon_0,M)}. 
\]
\item 
If $h(\frac{w}{2},\frac{w}{2})=0$, then $L^{\epsilon_0}_{p,M}$ is holomorphic. 

\item 
If $P_{N,p}(M) = P_H(M)$ and $h(\frac{w}{2},\frac{w}{2})=0$, then the holomorphic function $L_{p,M}^{\epsilon_0}$ 
is bounded. 
\end{itemize}
\end{conj}

\medskip
\subsubsection{\bf Motive attached to a modular form}
\label{sec:modular-motive}
Let $f \in S_k(N, \epsilon)$ be a primitive modular form for $k>1$ with Fourier coefficients in $\Q$ and let $M(f)$ be the Grothendieck motive attached to $f$ by Scholl \cite{scholl}. 
To ensure that $M(f)$ has coefficients in $\Q$, we assumed the Fourier coefficients 
of $f$ to be in $\Q$. 
This is a pure motive of weight $k-1$ with Hodge structure given by 
\[
\HH_B(M(f)) \otimes \C=\HH^{0,k-1} \bigoplus \HH^{k-1,0}.
\]
where both summands are $1$-dimensional. The Hodge polygon is the line segments joining 
$$
\{(0,0), \ (h^{(0,k-1)},0), \ (h^{(0,k-1)}+h^{(k-1,0)}, (k-1) h^{(k-1,0)})\} \ = \ 
\{(0,0), \ (1,0), \ (2,k-1)\}.
$$ 
We also have the following equality of $L$-functions 
 \[
 L(s,f \otimes \chi)=L(s, M(f) \otimes \chi).
 \]
The polynomial $Z_p(X,M(f))$ coincides with the Hecke polynomial (\ref{Hecke}). 
The $p$-Newton polygon for $M(f)$
 is a curve joining 
 $$
 \{(0,0), \ (1, v_p(a_p)), \ (2, v_p (\epsilon(p) p^{k-1}))\}.
 $$
Following Hida, we call a modular form $f$  to be  $p$-ordinary if $v_p(a_p)=0$.
Hence, a classical modular form is $p$-ordinary if and only if $M(f)$ is nearly $p$-ordinary.

\bigskip
\section{The symmetric power $L$-functions}
\label{sec:sym-powers}

\subsection{\bf Langlands functoriality for symmetric powers}
Let $\pi$ be a cuspidal 
automorphic representation of ${\rm GL}_2({\mathbb A})$, 
where $\A$ is the adele ring of $\Q.$ This means 
that, for some $s \in {\mathbb R}$, $\pi \otimes |\cdot |^s$ is an 
irreducible summand of
$
L^2_{\rm cusp}({\rm GL}_2(\Q)\backslash
{\rm GL}_2({\mathbb A}), \omega)
$
the space of square-integrable cusp forms with unitary central character $\omega$. 
We have the decomposition
$\pi = \otimes_p' \pi_p$ where $p$ runs over all places of $\Q$ and 
$\pi_p$ is an irreducible admissible representation of ${\rm GL}_2(\Q_p)$.
Given such a $\pi$, consider the Euler product of the standard (Jacquet--Langlands) $L$-function: 
$$
L(s, \pi) = \prod_p L_p(s, \pi_p), \quad \Re(s) \gg 0.
$$
For all $p$ outside a finite set $S$ of places including the archimedean ones and the places where $\pi$ is ramified,  suppose the Euler factor at $p$ looks like:
$$
L_p(s, \pi_p) = (1 - \alpha_p p^{-s})^{-1}(1 - \beta_p p^{-s})^{-1}. 
$$
For any Hecke character $\chi : \Q^\times \backslash \A^\times \to \C^\times$, we define a partial 
twisted $n$-th symmetric power $L$-function: 
$$
L^S(s, \Sym^n \otimes \chi, \pi) \ := \ \prod_p \prod_{j=0}^n (1 - \alpha_p^{n-j} \beta_p^j \chi(p)p^{-s})^{-1}, 
\quad \Re(s) \gg 0. 
$$
The Langlands program says that we should be able to complete this partial $L$-function at places $p \in S$ and
the completed $L$-function $L(s, \Sym^n \otimes \chi, \pi)$ is expected to have all the usual properties of analytic continuation, functional equation, etc. 

Let's elaborate a little further for which we recall the formalism of Langlands functoriality especially for symmetric powers. We will be brief here as there are 
several good expositions; 
see for instance Clozel~\cite{clozel2}.  
The local Langlands correspondence for ${\rm GL}_2$ (see \cite{kutzko} 
and \cite{kudla} for the $p$-adic case and \cite{knapp} for the 
archimedean case), 
says that to $\pi_p$ is associated a representation 
$\sigma(\pi_p) : W_{\Q_p}' \to {\rm GL}_2({\mathbb C})$ of the 
Weil--Deligne group $W_{\Q_p}'$ of $\Q_p$. (If $p$ is infinite, we take 
$W_{\Q_p}' = W_{\Q_p}$.) Let $n \geq 1$ be an integer. Consider the 
$n$-th symmetric power of $\sigma(\pi_p)$ which is an $n+1$ dimensional
representation. This is simply the composition of $\sigma(\pi_p)$ with 
${\rm Sym}^n : {\rm GL}_2({\mathbb C}) \to {\rm GL}_{n+1}({\mathbb C})$. 
Appealing to the local Langlands correspondence for ${\rm GL}_{n+1}$
(\cite{harris-taylor}, \cite{henniart}, \cite{knapp}, \cite{kudla})
we get an irreducible admissible representation of 
${\rm GL}_{n+1}(\Q_p)$ which we denote as ${\rm Sym}^n(\pi_p)$.  
Now define a global representation of ${\rm Sym}^n(\pi)$ of 
${\rm GL}_{n+1}({\mathbb A})$ by 
${\rm Sym}^n(\pi) := \otimes'_p \ {\rm Sym}^n(\pi_p).$
{\it Langlands principle of functoriality} predicts that 
${\rm Sym}^n(\pi)$ is an automorphic representation of ${\rm 
GL}_{n+1}({\mathbb A})$, i.e., it is isomorphic to an irreducible 
subquotient of the representation of ${\rm GL}_{n+1}({\mathbb A})$ 
on the space of automorphic forms \cite[\S4.6]{borel-jacquet}. 
If $\omega_{\pi}$ is the central character of $\pi$ then 
$\omega_{\pi}^{n(n+1)}$ is the central character of ${\rm Sym}^n(\pi)$. 
Actually it is expected to be an isobaric automorphic representation. 
(See \cite[Definition 1.1.2]{clozel2}
for a definition of an isobaric representation.) 
The principle of functoriality for the $n$-th symmetric power 
is known for $n=2$ by Gelbart--Jacquet \cite{gelbart-jacquet};
for $n=3$ by Kim--Shahidi \cite{kim-shahidi-annals2}; and 
for $n=4$ by Kim \cite{kim-jams}. For certain special forms $\pi$, for 
instance, if $\pi$ is dihedral then it is known for all $n$; see also Kim~\cite{kim-inventiones}. 
There has been recent breakthrough for higher symmetric powers by Clozel and Thorne 
\cite{clozel-thorne-comp} \cite{clozel-thorne-annals}.

The $n$-th symmetric power $L$-function of $\pi$ is expected to be the standard $L$-function for $\GL_{n+1}$ 
attached to the $n$-th symmetric power transfer $\Sym^n(\pi)$, i.e., 
$$
L(s, \Sym^n \otimes \chi, \pi) \ = \ L(s, \Sym^n(\pi) \otimes \chi). 
$$
For the standard $L$-function of $\GL$, see Jacquet~\cite{jacquet-corvallis}. 
We wish to understand the $p$-adic interpolation of the critical values of the 
$n$-th symmetric power $L$-function. There has been extensive work in the case of $n=2$; see, for example, 
Coates--Schmidt~\cite{coates-schmidt}, Schmidt~\cite{schmidt-inv}, and 
Dabrowski--Delbourgo~\cite{dabrowski-delbourgo}. {\it A relatively modest goal of this paper is 
to write down $p$-adic symmetric cube $L$-functions for $\GL_2$ while appealing to Langlands principle of functoriality.}

\medskip
\subsection{\bf Various approaches for symmetric cube $L$-functions}

We consider some approaches to lay one's hands on the twisted 
symmetric cube $L$-function $L(s, \Sym^3(\pi) \otimes \chi)$ attached to a cuspidal automorphic representation 
$\pi$ of $\GL_2$ over $\Q.$ Some of these will lead to a construction of the 
$p$-adic symmetric cube $L$-functions.

\subsubsection{\bf Via triple product $L$-functions}
The most natural environment to see symmetric cube is to consider triple products. Given a two-dimensional vector space $V$, it is easy to see that
$$
V \otimes V \otimes V \ = \ \Sym^3(V) \, \oplus \, (V \otimes \Lambda^2V) \, \oplus \, (V \otimes \Lambda^2V).
$$
Interpreting this in terms of Galois representations, via the local Langlands correspondence, we get the following equality 
of global $L$-functions: 
$$
L(s, \pi \times \pi \times \pi \otimes \chi) \ = \ 
L(s, \Sym^3(\pi) \otimes \chi) \, L(s, \pi \otimes \omega_\pi \chi)^2.
$$
 The $p$-adic $L$-function $L_p(s, \pi \otimes \omega_\pi \chi)$ has been described above. Furthermore, the triple product $p$-adic  $L$-functions have been studied by B\"ocherer and Panchishkin \cite{bocherer-panchishkin}. 
Putting the two together, one should be able to construct the $p$-adic symmetric cube $L$-function. 
Although we will not pursue this theme here, we will consider a very similar line of thought  
below.

\subsubsection{\bf Via the Langlands--Shahidi method}
If one considers the Langlands--Shahidi method, then one can see the symmetric cube $L$-functions as follows: Take a split reductive group $G$ of type ${\bf G}_2$, and consider the parabolic subgroup $P = MN$ where the Levi quotient 
$M$ has the shorter of the two simple roots and the unipotent radical $N$ has the root space corresponding to the longer of the simple roots. Then $M = \GL_2$, and the adjoint representation of $M$ on the Lie algebra of $N$ breaks up as 
$r_1 \oplus r_2$ where 
$r_1 = \Sym^3 \otimes {\rm det}^{-1}$ and $r _2 = {\rm det}.$ Given a cuspidal automorphic representation 
$\pi$ of $\GL_2$ the Langlands $L$-function $L(s, \pi, r_1)$ in this context is nothing but 
$L(s, \Sym^3\pi \otimes \omega_\pi^{-1}).$ (See Kim--Shahidi~\cite[\S\,1]{kim-shahidi-annals1} for more details.) 
However, with the current state of technology, it is not clear (to the authors) if the Langlands--Shahidi method is ready for $p$-adic interpolation.

\subsubsection{\bf Via $L$-functions for $\GL_4$ applied to $\Sym^3(\pi)$}
Using the Langlands principle of functoriality, a direct way to study $L(s, \Sym^3(\pi) \otimes \chi)$ is to study the 
standard $L$-function of $\GL_4 \times \GL_1$ applied to the representation $\Pi := \Sym^3(\pi)$ of $\GL_4$
which, as mentioned above, has been proven to be an automorphic representation, and the twisting character 
$\chi$ which is on $\GL_1.$ This representation $\Pi$ admits what is called a Shalika model, and in such a situation there is a construction of $p$-adic $L$-function due to Ash--Ginzburg \cite{ash-ginzburg}. We will explicate this method 
in Sect.~\ref{sec:lfnGL4} below.

\subsubsection{\bf Via $L$-functions for $\GL_3 \times \GL_2$ applied to $\Sym^2(\pi) \times \pi$}
\label{sec:l-fn-gl3-gl2}
Given a two-dimensional vector space $V$, it is easy to see that
$$
\Sym^2(V) \otimes V  \ = \ \Sym^3(V) \, \oplus \, (V \otimes \Lambda^2V).
$$
Interpreting this in terms of Galois representations, via the local Langlands correspondence, we get the following equality 
of global $L$-functions: 
$$
L(s, \Sym^2(\pi) \times \pi \otimes \chi) \ = \ 
L(s, \Sym^3(\pi) \otimes \chi) \, L(s, \pi \otimes \omega_\pi \chi).
$$
For the left hand side, there has been an extensive study of arithmetic properties of critical values for $L$-functions 
of $\GL_n \times \GL_{n-1}$. See, for example,  \cite{janus-crelle}, \cite{janus-imrn}, \cite{janus-preprint}, 
\cite{kasten-schmidt}, \cite{kms}, \cite{mahnkopf},  \cite{raghuram-imrn}, \cite{raghuram-preprint}, and \cite{schmidt-inv2}. 
To study the $p$-adic interpolation of these critical values, amongst the above references, 
Schmidt~\cite{schmidt-inv2} and Januszewski~\cite{janus-preprint} are particularly relevant. We will explicate this theme in 
\S\,\ref{sec:gl3-gl2} below.

\medskip
\subsection{\bf Cuspidality criterion for symmetric power transfers}
To study arithmetic properties of symmetric power $L$-functions as suggested above, we need to know certain properties of the symmetric power transfers. To begin, we recall the cuspidality criterion for the 
symmetric cube transfer due to Kim and Shahidi~\cite{kim-shahidi-duke}. 

\begin{thm}
\label{thm:kim-shahidi}
Let $\pi$ be a cuspidal automorphic representation of $\GL_2$ over a number field $F$. Then
\begin{enumerate}
\item (Dihedral Case) If $\pi = \pi \otimes \nu$ for some (necessarily quadratic) nontrivial character $\nu$, 
then $\nu$ corresponds to a quadratic extension $E/F$ and $\pi = \pi(\chi)$ the automorphic induction of a Hecke character 
$\chi$ of $E$. In this case, $\Sym^r(\pi)$ is not cuspidal for any $r \geq 2$. The precise isobaric decomposition of $\Sym^3(\pi)$ depends on 
whether $\chi \chi'^{-1}$ factors through the norm map from $E$ to $F$. (See \cite[\S\,2.1]{kim-shahidi-duke}.) 

\smallskip

\item (Tetrahedral Case) 
If $\pi \neq \pi \otimes \nu$ for any $\nu$, but $\Sym^2(\pi) = \Sym^2(\pi) \otimes \mu$ for 
some (necessarily cubic) nontrivial character $\mu$, then 
$$
\Sym^3(\pi) \ = \ (\pi \otimes \omega_\pi\mu) \, \oplus \, (\pi \otimes \omega_\pi\mu^2). 
$$

\smallskip

\item If $\pi \neq \pi \otimes \nu$ and $\Sym^2(\pi) \neq \Sym^2(\pi) \otimes \mu$ for any nontrivial characters $\nu$ or 
$\mu$, then $\Sym^3(\pi)$ is cuspidal. 
\end{enumerate}
\end{thm}
See \cite[Thm.\,2.2.2]{kim-shahidi-duke}. See also the discussion of the various polyhedral types towards the end of 
\S\,3.3 in {\it loc.\,cit.} In short, we may write 
\begin{equation}
\begin{split}
\Sym^2(\pi) \ \mbox{is cuspidal}  & \ \iff \ \pi \ \mbox{is not dihedral, and} \\
\Sym^3(\pi) \ \mbox{is cuspidal}  & \ \iff \ \pi \ \mbox{is neither dihedral nor tetrahedral.} 
\end{split}
\end{equation}

\medskip
\subsection{\bf The property of being cohomological for symmetric power transfers}

To put ourselves in an arithmetic context, we need to work with representations which contribute to cohomology. 
We quote the following theorem proved in \cite{raghuram-preprint} that a symmetric power transfer of a cohomological representation is again of 
cohomological type. For this section, we follow the notations as in \cite{raghuram-preprint}. 

Let $\mu \in X^+(T_2)$ be a dominant integral weight for $\GL_2/\Q$ and let $M_\mu$ be the finite-dimensional irreducible representation of $\GL_2(\C)$ of highest weight $\mu.$  
Suppose $\mu = (a, b) \in \Z^2$ with $a \geq b.$  Define a weight 
$\Sym^r(\mu) \in X^+(T_{r+1})$ as $\Sym^r(\mu) \ := \ (ra, (r-1)a + b, \dots, a + (r-1)b, rb).$
If ${\sf w} = {\sf w}(\mu) = a + b$ be the purity weight of $\mu$, then it is easy to check that $\Sym^r(\mu)$ is also pure 
(see \cite{raghuram-preprint} for purity), 
and it's purity weight is ${\sf w}(\Sym^r(\mu)) = r {\sf w}.$

\begin{thm}
\label{thm:sym-coh}
Let $\mu \in X^+(T_2)$ and $\pi \in \Coh(G_2, \mu^{\sf v})$, i.e., $\pi$ is a cuspidal automorphic representation of 
$\GL_2(\A)$ such that $\pi_\infty \otimes M_\mu^{\sf v}$ has nontrivial relative Lie algebra cohomology. 
Suppose $\Sym^r(\pi)$ is a cuspidal automorphic representation of $G_{r+1}$, then 
$\Sym^r(\pi) \in \Coh(G_{r+1}, \Sym^r(\mu)^{\sf v}).$
\end{thm}

\medskip
\subsection{\bf Near ordinarity of symmetric powers of a modular motive $M(f)$}
Recall from \S\,\ref{sec:modular-motive} the pure motive $M(f)$ of weight $k-1$ attached 
to an eigenform $f \in S_k(N,\epsilon).$ The Hodge numbers of $M(f)$ are 
$h^{(0,k-1)}=h^{(k-1,0)}=1.$
If the roots of the Hecke polynomial at $p$ of $M(f)$ 
are $\alpha$ and $\beta$, then the roots of the Hecke polynomial of $\Sym^3(M(f))$ are 
$\alpha^3$, $\alpha^2 \beta$, $\alpha\beta^2$ and $\beta^3$. 
The following proposition asserts that if $M(f)$ is nearly $p$-ordinary then the  
motive $\Sym^3(M(f))$ (see Deligne \cite{Deligne}) attached to $\Sym^3$ transfer of automorphic representation $M(f)$ is also nearly $p$-ordinary. 

\begin{prop}
\label{nearsym}
If $M(f)$ is nearly $p$-ordinary  then $\Sym^3(M(f))$ is also nearly $p$-ordinary.
\end{prop}
\begin{proof}
If $\alpha$ and $\beta$ are roots of the Hecke polynomial at $p$ for $f$, 
then 
the roots of the Hecke polynomial at $p$ of $\Sym^3(M(f))$ are 
$\alpha^3$, $\alpha^2 \beta$, $\alpha \beta^2$ and $\beta^3$.  Recall, $M(f)$ is nearly $p$-ordinary 
if and only if $v_p(a_p)=0$. 
For $\Sym^3(M(f))$, the coefficients of the Hecke polynomial are 
\begin{eqnarray*}
A_1 \ = \ & a_p(a_p^2-2 \epsilon(p)p^{k-1}),  \\
A_2 \ = \ & \epsilon(p)p^{k-1}[a_p^2-2(\epsilon(p)p^{k-1})^2+a_p^2(\epsilon(p)p^{k-1})^2], \\ 
A_3 \ = \ & (\epsilon(p)p^{k-1})^3[(\alpha+\beta)^3-3 \alpha_p \beta(\alpha+\beta)+\alpha^2\beta^2)], \\
A_4 \ = \ & \alpha^6\beta^6
\end{eqnarray*}
For any two element $\alpha$ and $\beta$ of $\C_p$ with $v_p(\alpha) \neq v_p(\beta)$, we have
$v_p(\alpha+\beta)=\mathrm{min}(v_p(\alpha), v_p(\beta))$.  
A small check shows that $v_p(A_1)=0$, $v_p(A_2)=k-1+2 v_p(a_p)=k-1$,  $v_p(A_3)=3(k-1)$
and  $v_p(A_4)=6(k-1)$. The $p$-Newton polygon of $\Sym^3(M(f))$ consists of the line segments joining the points
$$
(0,0), \ (1,0),  \ (2, k-1), \ (3, 3k-3), \ (4, 6k-6).
$$
Next, the Hodge types of $\Sym^3(M(f))$ are 
$(0, 3(k-1)), ((k-1), 2(k-1)), (2(k-1), (k-1)), (3(k-1),0)$ and all the nonzero Hodge numbers are $1.$ Hence, 
the Hodge polygon of $\Sym^3(M(f))$ also consists of the line segments joining the same set of points: 
$(0,0), \, (1,0),  \, (2, k-1), \, (3, 3k-3), \, (4, 6k-6)$. Hence,  $\Sym^3(M(f))$ is nearly $p$-ordinary.
\end{proof}

\begin{remark}
The same method can be applied to show that for any 
$m \geq 4$  the motive $\Sym^m(M(f))$ is nearly $p$-ordinary if $M(f)$ is nearly $p$-ordinary. For $\Sym^m(M(f))$, the roots of the Hecke polynomial at $p$ are 
$\alpha^m, \alpha^{m-1} \beta,\, \cdots,\,\beta^m$ if $\alpha$ and $\beta$ are roots of the Hecke polynomial at $p$ of $f$. The $p$-Newton polygon of $\Sym^m(M(f))$ consists of the line segments joining 
$$
(0,0), \ (1,0),  \ (2, k-1), \ (3, 3k-3), \ (4, 6k-6), \dots, (m+1, (k-1)(1+2+3+ \cdots +m)).
$$
The Hodge types of $\Sym^m(M(f))$ are 
$(0, m(k-1)), ((k-1), (m-1)(k-1)),\, \cdots ,\, (m(k-1),0)$ and all the nonzero Hodge numbers are $1.$ 
The Hodge polygon of $\Sym^m(M(f))$  is the line segments joining 
$(0,0),\,(1,0),\, (2, k-1), \, (3, 3k-3),\,\cdots,\,(m+1, \frac{(k-1)m(m+1)}{2})$. Hence, the  conjectural motive $\Sym^m(M(f))$ is nearly $p$-ordinary for all $m$.
\end{remark}

\bigskip
\section{$p$-adic $L$-functions for $\GL_4$}
\label{sec:lfnGL4}

\medskip
\subsection{\bf Shalika models and $L$-functions for ${\rm GL}_4$}
\label{sect:globalShalikamodels}
The following is a summary of a certain analytic theory of the standard $L$-function attached to a cuspidal
automorphic representation $\Pi$ of $\GL_4$ over $\Q$ which admits a Shalika model. 
The presentation is based on \cite[\S\,3.1--3.3]{grobner-raghuram}. Since we want to focus on the symmetric 
cube $L$-function, we will exclusively work with $\GL_4$ in this section. 

\medskip
\subsubsection{\bf Global Shalika models and exterior square $L$-functions}
Let
$$ 
S:=\left\{
s =  \left( \!\!\begin{array}{ccc}
h &  0 \\
0 &  h
\end{array}\!\!\right)
\left( \!\!\begin{array}{ccc}
1 &  X\\
0 &  1
\end{array}\!\!\right) \Bigg|
\begin{array}{l}
h\in \GL_2\\
X\in {\rm M}_2
\end{array}\right\}\subset G =: \GL_{4}.
$$
It is traditional to call $S$ the Shalika subgroup of $G$. Let $\psi : \Q \backslash \A \to \C^\times$ be a 
nontrivial additive character which is fixed once and for all. Let $\eta : \Q^\times \backslash \A^\times \to \C^\times$ 
be a Hecke character of $\Q$. These characters can be extended to a character of $S(\A)$:
$$
s=\left(\!\! \begin{array}{ccc}
h &  0\\
0 &  h
\end{array}\!\!\right)\left(\!\! \begin{array}{ccc}
1 &  X\\
0 &  1
\end{array}\!\!\right) \mapsto (\eta \otimes \psi)(s) := \eta(\det(h))\psi(Tr(X)).
$$
We will also denote $\eta(s) = \eta(\det(h))$ and $\psi(s) = \psi(Tr(X))$. 

Let $\Pi$ be a cuspidal automorphic representation of $\GL_4/\Q.$ Assume that $\eta^2 = \omega_\Pi$. 
For a cusp form $\varphi\in \Pi$ and $g\in G(\A)$, consider the integral
$$
S^\eta_\psi(\varphi)(g):=\int_{Z_{G}(\A)S(F)\backslash S(\A)} (\Pi(g)\cdot\varphi)(s)\eta^{-1}(s)\psi^{-1}(s) ds.
$$
It is well-defined and hence yields a function
$\mathcal{S}^\eta_\psi(\varphi): G(\A)\rightarrow\C$ satisfying 
$S^\eta_\psi(\varphi)(sg) = \eta(s)\cdot\psi(s)\cdot S^\eta_\psi(\varphi)(g),$
for all $g\in G(\A)$ and $s\in S(\A)$. The following theorem 
due to Jacquet and Shalika~\cite[Thm.\,1]{jacshal}
gives a necessary and sufficient condition for 
$S^\eta_\psi$ being non-zero.

\begin{thm}
\label{thm:JS}
The following assertions are equivalent:
\begin{enumerate}
\item[(i)] There is a $\varphi\in\Pi$ and $g\in G(\A)$ such that $S^\eta_\psi(\varphi)(g)\neq 0$.
\item[(ii)] $S^\eta_\psi$ defines an injection of $G(\A)$-modules
$
\Pi\hookrightarrow \textrm{\emph{Ind}}_{S(\A)}^{G(\A)}[\eta\otimes\psi].
$
\item[(iii)] Let $S$ be any finite set of places containing $S_{\Pi,\eta}$.
The twisted partial exterior square $L$-function
$$
L^S(s,\Pi,\wedge^2\otimes\eta^{-1}):=\prod_{v\notin S} L(s,\Pi_v,\wedge^2\otimes\eta^{-1}_v)
$$
has a pole at $s=1$.
\end{enumerate}
\end{thm}

\begin{definition}
\label{defn:shalika-model}
If $\Pi$ satisfies any one, and hence all, of the equivalent conditions of Thm.\ \ref{thm:JS}, then we say that 
$\Pi$ {\it has an $(\eta,\psi)$-Shalika model}, and
we call the isomorphic image $S^\eta_\psi(\Pi)$ of $\Pi$ under $S^\eta_\psi$ a {\it global $(\eta,\psi)$-Shalika model} of $\Pi$. 
We will sometimes suppress the choice of the characters $\eta$ and $\psi$ and simply say that $\Pi$ has a Shalika model.
\end{definition}

\medskip
\subsubsection{{\bf Period integrals over $\GL_2 \times \GL_2$}}
The following proposition, due to Friedberg and Jacquet~\cite[Prop.\,2.3]{friedjac}, 
relates the period-integral over $H :=  \GL_2 \times \GL_2 \subset G$ 
of a cusp form $\varphi$ of $G(\A)$ to a certain zeta-integral of the
function $S^\eta_\psi(\varphi)$ in the Shalika model corresponding to $\varphi$ over one copy of $\GL_2$ in $H.$

\begin{prop}\label{prop:FJ}
Let $\Pi$ have an $(\eta,\psi)$-Shalika model. For a cusp form $\varphi\in\Pi$, consider the integral
$$
\Psi(s,\varphi):=\int_{Z_G(\A)H(\Q)\backslash H(\A)} \varphi\left(\!\!\left(\!\! \begin{array}{ccc}
h_1 &  0\\
0 &  h_2
\end{array}\!\!\right)\!\!\right)\Bigg|\frac{\det(h_1)}{\det(h_2)}\Bigg|^{s-1/2}\eta^{-1}(\det(h_2))\, d(h_1,h_2).
$$
Then, $\Psi(s,\varphi)$ converges absolutely for all $s\in\C$. Next, consider the integral
$$
\zeta(s,\varphi):=\int_{\GL_n(\A)} S^\eta_\psi(\varphi)\left(\!\!\left(\!\! \begin{array}{ccc}
g_1 &  0\\
0 &  1
\end{array}\!\!\right)\!\!\right) |\det(g_1)|^{s-1/2} \, dg_1.
$$
Then, $\zeta(s,\varphi)$ is absolutely convergent for $\Re(s)\gg 0$.
Further, for $\Re(s)\gg 0$, we have
$$
\zeta(s,\varphi)=\Psi(s,\varphi),
$$
which provides an analytic continuation of $\zeta(s,\varphi)$ by setting $\zeta(s,\varphi)=\Psi(s,\varphi)$ for all $s\in\C$.
\end{prop}

\medskip
\subsubsection{{\bf Local Shalika models}}
Consider a cuspidal automorphic representation $\Pi=\otimes'_p \Pi_p$ of $G(\A)$.
\begin{definition}
For any place $p$ we say that $\Pi_p$ has a local $(\eta_p,\psi_p)$-Shalika model if there is a non-trivial (and hence injective) intertwining
$\Pi_p \hookrightarrow \textrm{Ind}_{S(\Q_p)}^{G(\Q_p)}[\eta_p \otimes \psi_p].$
\end{definition}

If $\Pi$ has a global Shalika model, then $S^\eta_\psi$ defines local Shalika models at every place. The 
corresponding local intertwining operators are denoted by $S^{\eta_p}_{\psi_p}$ and their images by $S^{\eta_p}_{\psi_p}(\Pi_p)$,
whence $S^\eta_\psi(\Pi)=\otimes'_p S^{\eta_p}_{\psi_p}(\Pi_p)$.
We can now consider cusp forms $\varphi$ such that the function $\xi_\varphi= S^\eta_\psi(\varphi)\in S^\eta_\psi(\Pi)$ is factorizable 
as $\xi_\varphi=\otimes'_p \xi_{\varphi_p},$
where
$$
\xi_{\varphi_p}\in S^{\eta_p}_{\psi_p}(\Pi_p)\subset \textrm{Ind}_{S(\Q_p)}^{G(\Q_p)}[\eta_p\otimes\psi_p].
$$
Prop.\,\ref{prop:FJ} implies that
$$
\zeta_p(s,\xi_{\varphi_p}):=\int_{\GL_2(\Q_p)} \xi_{\varphi_p}\left(\!\!\left(\!\! \begin{array}{ccc}
g_{1,p} &  0\\
0 &  1_p
\end{array}\!\!\right)\!\!\right) |\det(g_{1,p})|^{s-1/2} dg_{1,p}
$$
is absolutely convergent for $\Re(s)$ sufficiently large. The same remark applies to
$$
\zeta_f(s,\xi_{\varphi_f}):=\int_{\GL_2(\A_f)} \xi_{\varphi_f}\left(\!\!\left(\!\! \begin{array}{ccc}
g_{1,f} &  0\\
0 &  1_f
\end{array}\!\!\right)\!\!\right) |\det(g_{1,f})|^{s-1/2} dg_{1,f}=\prod_{p \neq \infty} \zeta_p(s,\xi_{\varphi_p}).
$$

\medskip
\subsubsection{\bf Shalika-zeta-integral and the standard $L$-function of $\Pi$}

See Friedberg and Jacquet \cite[Prop.\,3.1, 3.2]{friedjac} for the following proposition: 

\begin{prop}
\label{prop:FJL-fct}
Assume that $\Pi$ has an $(\eta,\psi)$-Shalika model. Then for each place $p$ and $\xi_{\varphi_p}\in S^{\eta_p}_{\psi_p}(\Pi_p)$ 
there is a holomorphic function $P(s,\xi_{\varphi_p})$ such that
$$
\zeta_p(s,\xi_{\varphi_p})=L(s,\Pi_p)P(s,\xi_{\varphi_p}).
$$
One may hence analytically continue $\zeta_p(s,\xi_{\varphi_p})$ by re-defining it to be $L(s,\Pi_p)P(s,\xi_{\varphi_p})$ for all $s\in\C$.
Moreover, for every $s\in\C$ there exists a vector $\xi_{\varphi_p}\in S^{\eta_p}_{\psi_p}(\Pi_p)$ such that $P(s,\xi_{\varphi_p})=1$. 
If $p \notin S_\Pi$, then this vector can be taken to be the spherical vector $\xi_{\Pi_p}\in S^{\eta_p}_{\psi_p}(\Pi_p)$ normalized by the condition
$$\xi_{\Pi_p}(id_p)=1.$$
\end{prop}

\medskip
\subsection{\bf The unramified calculation}
\label{sec:unramified-calc}
Let $\nu_1, \dots, \nu_4$ be unramified 
characters of $\Q_p^\times$  and let $\nu = \nu_1\otimes \cdots \otimes \nu_4$  
be the character on the diagonal torus $T = T_4(\Q_p)$ of $G = \GL_4(\Q_p)$. Let $B = TU$ be the subgroup of all upper triangular matrices in $G$ and suppose $\delta_B$ is the modular character of $B.$ 
Assume that the representation 
$$
\nu_1 \times \cdots \times \nu_4 \ := \ {\rm Ind}_B^G(\nu_1 \otimes \cdots \otimes \nu_4)
$$ 
obtained by normalized parabolic induction is irreducible. Then it is an irreducible, unramified, and generic representation. (We are only interested in local components of a global cuspidal representation.) 
In \cite[Prop.\,1.3]{ash-ginzburg}, it is proved that $\nu_1 \times \cdots \times \nu_4$ 
admits an $(\eta, \psi)$-Shalika model  
if and only if up to a permutation of $\{\nu_1,\dots, \nu_4\}$ we have $\nu_1 \nu_3 = \nu_2\nu_4 = \eta.$ 

\smallskip

Let $\Pi$ be an irreducible cuspidal representation of $\GL_4(\A)$ with trivial central character; suppose that 
$\Pi$ is of cohomological type with respect to the trivial coefficient system, i.e., $\Pi \in \Coh(\GL_4, \mu)$ with 
$\mu = 0.$ Suppose also that $\Pi$ admits a Shalika model. One expects a cohomological cuspidal $\Pi$ to correspond to a motive $M(\Pi)$ satisfying the relation 
\[
L(s, M(\Pi) \otimes \chi)=L(s-\frac{1}{2}, \Pi \otimes \chi).
\]
({\bf Note:} The above normalization $M(\Pi)$ is what is used in \cite{ash-ginzburg} and so we stick to it; however, the reader is referred to Clozel~\cite{clozel2} for a more commonly used normalization wherein one has 
$L(s, M(\Pi))=L(s-\tfrac{(n-1)}{2}, \Pi)$ for a cuspidal representation $\Pi$ of $\GL_n/\Q$ of motivic type. For another normalization in terms of an effective motive, see \cite{harder-raghuram}.)

\begin{prop}
\label{nearlyordinarygl4}
Let $\Pi$ be as above, and take an unramified prime $p.$ Then the local component $\Pi_p$ if of the form 
$\Ind_B^G(\nu) $ with $\nu=\nu^{-1}_2 \times \nu^{-1}_1  \times \nu_1 \times \nu_2$.
Suppose we have 
$$
v_p(\nu_i(p))=i-\frac{5}{2}, \quad i = 1,2. 
$$
then $M(\Pi)$ is nearly $p$-ordinary. 
\end{prop}

\begin{proof}
Recall, $\Pi$ is motivic and  the corresponding motive $M(\Pi)$ has weight 
$-1$ and rank $4$. The Hodge decomposition of $M(\Pi)$ is 
\begin{equation}
\label{hodge}
\HH_B(M(\Pi)) \otimes \C=\HH^{(-2,1)} \oplus \HH^{(-1,0)} \oplus \HH^{(0,-1)} \oplus \HH^{(1,-2)} 
\end{equation}
with each factor $1$-dimensional. Let $\alpha_1, \alpha_2, \alpha_3$ and $\alpha_4$ 
be roots of the Hecke polynomial at $p$.  From \S\,\ref{Motive}, the $p$-adic valuations $v_p$ 
of the coefficients of Hecke polynomials at $p$ are
$v_p(A_1)=-2$, $v_p(A_2)=-3$, $v_p(A_3)=-3$ and $v_p(A_4)=-2$. 
Hence, the $p$-Newton polygon is the line segments joining $(0,0)$, $(1, -2)$, $(2,-3)$, $(3,-3)$ 
and $(4,-2)$.  The Hodge polygon is also the line segments 
joining $(0,0)$, $(1, -2)$, $(2,-3)$, $(3,-3)$ 
and $(4,-2)$ which can be seen from (\ref{hodge}). Hence, $M(\Pi)$ is nearly $p$-ordinary.  
\end{proof}

Under the above conditions satisfied by $\nu_i(p)$, we will say that $\Pi$ is nearly $p$-ordinary. 
Set
\[
\lambda \ = \ p^2 \nu_1(p) \nu_2(p).
\]
Observe that  $\lambda$ is a $p$-adic unit, since $v_p(\lambda)=2+1-\frac{5}{2}+2-\frac{5}{2}=0$.

\medskip
\subsection{\bf A special choice of a cusp form $\phi$}
\label{Assumption}
Let $\Pi$ be as in \S\,\ref{sec:unramified-calc}. Let $p$ be an unramified place and put 
$S = \{\infty, p\}.$ We make a special choice of a vector $\phi=\otimes' \phi_l$ in the space of $\Pi$. 

\begin{itemize}
 \item $l = p.$ We will take $\phi_p$ to be a very special Iwahori spherical vector.  
Let $\cI_p$ be the standard Iwahori subgroup of $\GL_4(\Z_p)$ consisting of all matrices  which are upper-triangular modulo $p$. 
In the induced representation $\nu_1 \times \cdots \times \nu_4$, we write down a special Iwahori spherical vector: 
$$
F_\nu(g) \ = \ \left\{ \begin{array}{ll} 
\delta_B^{1/2}(b) \nu(b), & \mbox{if $g = bw_0 k \in Bw_0\cI$, and } \\
0, & \mbox{if not,}
\end{array}\right.
$$
where $w_0$ is the element of the Weyl group of longest length. From the induced model we map into the Shalika model 
via the integral 
\[
H_{f_{\nu}}(h)=\int_{B_2 \backslash \GL_2} \int_{M_2} f_{\nu}[( \begin{smallmatrix}
 & I \\
I &  
\end{smallmatrix})(\begin{smallmatrix}
I & X  \\
 &  I
\end{smallmatrix})(\begin{smallmatrix}
g &  \\
 &  g
\end{smallmatrix})h]\eta^{-1}(g) \psi(tr(X))dXdg, 
\]
where we have currently adopted local notations. If $\nu \in \Omega := \{\nu \ | \ |\nu_i \nu_j(p)|<1, 1 \leq i,j \leq 2 \}$, then $\{H_{f_\nu} \ | \  f_\nu \in {\rm Ind}_B^G(\nu \delta_B)^{\frac{1}{2}}\}$ defines a Shalika model for $\Pi_p$. 
We will take $\phi_p$ to be such that in the Shalika model it corresponds to $H_{F_{\nu}}.$

\smallskip

\item  $l \notin S.$ 
Choose $\phi_l$ such that the local zeta integral of the corresponding Shalika vector is the local $L$-factor, i.e., 
choose $\phi_l$ such that
$\zeta_l(s,H_{\phi_l}, \chi_l)=L(s, \Pi_l \otimes \chi_l)$  
for any character $\chi_l$ of $\Q_l^{\times}$.  This is possible due to \cite[Prop.\,3.1, 3.2]{friedjac}. 
 
\smallskip 
 
\item $l = \infty.$
Choose any cohomological $\phi_{\infty}$. (This is a delicate point which we will elaborate further below.) 

\end{itemize}

\medskip
\subsection{\bf Period integrals and a distribution on $\Z_p^\times$}

\begin{definition}
\label{fourdimidele}
%[Id\'ele cone for $\GL_4$]
For a positive integer $m \geq 1$ and $\epsilon \in \Z_p^{\times}$, set  $f=p^m$ and 
\[
C_{\epsilon,f}^* \ = \ 
\left\{ \left(\begin{array}{ll}
g_1 & 0 \\
0 &  g_2
\end{array} \right) \in \GL_2(\A) \times \GL_2(\A) \ | \ 
\det(g_1 g_2^{-1}) \in  \Q^{\times} \cdot ((\Q_{\infty}^{\times})^0 (\prod_{l \neq p} \Z_l^{\times})(\epsilon+f\Z_p)) \right\}.
\]
Define the {\it id\'ele cone of conductor $f$} as 
$$
C_{\epsilon,f}=Z(\A)(\GL_2(\Q) \times \GL_2(\Q)) \backslash C_{\epsilon,f}^*.
$$
For  an element $A$ of $M_2(\Z_p)$, set
\[
P(A, f)=\int_{C_{1,f}} \phi((\begin{smallmatrix}
g_1 &  \\
0 &  g_2
\end{smallmatrix}\ ) (\begin{smallmatrix}
I &  Af^{-1}\\
 &  I
\end{smallmatrix}\ )) \eta^{-1}(g_2) dg_1 dg_2, 
\]
where $\phi$ is the special cusp form chosen in \S\,\ref{Assumption}. 
\end{definition}

For the following proposition see Ash--Ginzburg~\cite[Prop.\,2.3]{ash-ginzburg}. The proof of this proposition  involves checking many formal properties of the above period integrals. 

\begin{prop}
\label{dist}
Recall, $\lambda = p^2 \nu_1(p) \nu_2(p)$, and put $\kappa = p^{-4}\lambda,$ i.e., 
$\kappa =  p^{-2} \nu_1(p) \nu_2(p).$
Define a function $\mu_\Pi$ on certain open subsets of $\Z_p^{\times}$ by 
\begin{eqnarray*}
\mu_\Pi(a+f\Z_p) \ = \ \kappa^{-m} P(diag(a,1),f),  & \mbox{if $m \geq 1,$ and} \\ 
& \\
\mu_\Pi(\Z_p^\times) \ = \ \sum_{a \in (\Z/p)^\times} \mu_\Pi(a + p\Z_p). &  
\end{eqnarray*}
Then $\mu_\Pi$ is a distribution on $\Z_p^{\times}$. 
\end{prop}

If $\Pi$ is nearly $p$-ordinary, then the quantity $\lambda$ above is a $p$-adic unit. In this case, 
Ash and Ginzburg \cite[\S\,5.3]{ash-ginzburg} 
prove that the distribution $\mu_\Pi$ is in fact a measure by proving the following proposition:

\begin{prop}
If $\Pi$ is cohomological (with respect to the trivial coefficient system) 
cuspidal with trivial central character and admitting a Shalika model, and suppose 
$\Pi$ is nearly $p$-ordinary (and hence $\lambda$ is a $p$-unit) then the values of $\mu_{\Pi}$ lie in a finitely generated
$\Z_p$-submodule of $\C_p$.
\end{prop}
We note that this will ensure that the distribution $\mu_{\Pi}$ is bounded
since the maximum valuation of the finite number of bounded 
numbers are finite and the elements of $\Z_p$ are of bounded valuations.

\medskip
\subsection{\bf Interpolation of $L(\tfrac12, \Pi \otimes \chi)$}
The above $p$-adic measure $\mu_\Pi$ gives a $p$-adic $L$-function by taking Mellin transforms.  
These $p$-adic $L$-functions interpolate critical values of the complex $L$-functions of the automorphic representations $\pi$ on $\GL_4(\A)$. Let $S$ be a set of finite places of $\Q$, we define $L_S(s, \Pi)=\prod_{l \in S} L(s, \Pi_l)$ and 
$L^S(s, \Pi)=\prod_{ l \not\in S} L(s, \Pi_l)$. For a proof of the following theorem, see Ash--Ginzburg~\cite[\S\,2.2]{ash-ginzburg}.

\begin{thm}
\label{Ash-Ginzburg}
Let  $\chi=\prod_l \chi_l$  denote a Hecke character of $\Q$ of finite order (i.e., it is the adelization of a classical Dirichlet character), unramified outside of $p$, trivial at infinity and with conductor $m > 1$ at $p$.
Let $\Pi$ be a cohomological (with respect to the trivial coefficient system) cuspidal automorphic representation of $\GL_4/\Q$ with trivial central character, which is nearly $p$-ordinary and for which $s=1/2$ is critical for $L(s, \Pi).$ 
Furthermore, assume that there exists a character $\chi'$, with the same properties as $\chi$ above, such that
$
L^S(\frac{1}{2}, \Pi \otimes \chi') \neq 0.$
Then the distribution $\mu_{\Pi}$ defined above is nonzero and we have
\[
\int_{\Z_p^{\times}} \chi_p(a) d \mu_{\Pi}(a) \ = \ 
c' \lambda^{-m} p^{2m} G(\chi_p) L^S(\frac{1}{2}, \Pi \otimes \chi)
\]
where $c'$ is a nonzero constant independent of $\chi$.
\end{thm}

\medskip
\subsection{\bf $p$-adic symmetric cube $L$-function -- I}
\label{symcube1}
Let $f \in S_k(N,\epsilon)$ be a classical elliptic eigen-cusp-form with Fourier coefficients in $\Q$ and let $\pi(f)$ 
be the corresponding automorphic representation. Suppose that $k \geq 2$ (see below).  
By Thm.\,\ref{Ash-Ginzburg}, a $p$-adic $L$-function for $Sym^3(\pi)$ exists if the automorphic representation $Sym^3(\pi)$ satisfies the following conditions: 

\begin{itemize}
 \item {\it  $\Sym^3(\pi)$ is cuspidal.} This follows from Thm.\,\ref{thm:kim-shahidi} by assuming that $f$ is not dihedral;  Since $k \geq 2$, the form $f$ or the representation $\pi$ is not tetrahedral (see, for example, \cite[Rem.\,3.8]{raghuram-shahidi-aim}).  
 
\smallskip

  \item {\it $\Sym^3(\pi)$ is cohomological (with respect to the trivial coefficient system).} This follows from 
 Thm.\,\ref{thm:sym-coh} provided we take $k=2$, because then $\pi$ would have cohomology with respect to the trivial  coefficient system and then so would $\Sym^3(\pi)$.

\smallskip  
  
  \item {\it $\Sym^3(\pi)$ is unramified and nearly ordinary at $p$.}
If we take $f$, or equivalently $\pi,$ to be unramified and nearly ordinary at $p$ then, by Prop.\,\ref{nearsym}, 
$\Sym^3(\pi)$ is also unramified and nearly $p$-ordinary. 

\smallskip

   \item {\it $\Sym^3(\pi)$ has trivial central character and admits a Shalika model.} 
 We know from Kim~\cite{kim-jams} that 
 $$
 \wedge^2(\Sym^3(\pi)) \ = \ \Sym^4(\pi) \otimes \omega_\pi \, \boxplus \, \omega_\pi^3.
 $$  Using (iii) of Thm.\,\ref{thm:JS} we see that $\Sym^3(\pi)$ has a Shalika model with $\eta = \omega_\pi^3.$
 Now the central character of $\pi$ is the nebentypus character $\epsilon$ of $f$. Hence, if we take $f$ such that 
 $\epsilon$ is a cubic character then $\Sym^3(\pi)$ has a Shalika model with 
 $\eta$ the trivial character; furthermore, $\omega_{\Sym^3(\pi)} = \omega_\pi^6$ which is also trivial.

\smallskip   
   
 \item  {\it $L(\frac{1}{2}, \Sym^3(\pi) \otimes \chi') \neq 0$ for a Hecke character $\chi'$.} Such a result on nonvanishing of twists is not available at the moment for representations of $\GL_4$ at $s=1/2.$ However, Ash and Ginzburg need this
assumption to ensure that a certain quantity coming from archimedean considerations (that involves the choice of cohomological vector $\phi_\infty$) is nonvanishing. This latter nonvanishing is now guaranteed by a result of Sun~\cite{sun}. 
\end{itemize}

To summarize, the above theorem of Ash and Ginzburg gives a $p$-adic symmetric cube $L$-function for a holomorphic cusp form $f \in S_k(N,\epsilon)$ only when $f$ is not dihedral, $k=2$, $\epsilon$ is a cubic character, and $f$ is nearly ordinary at $p$. The reader is referred to the forthcoming \cite{dimitrov-januszewski-r}, where using 
the results of \cite{grobner-raghuram} and generalizations of the modular symbols as in Dimitrov~\cite{dimitrov}, $p$-adic symmetric cube $L$-functions are constructed for a Hilbert modular form of arithmetic type with none of the above restrictions.

\bigskip
\section{$p$-adic $L$-functions for $\GL_3 \times \GL_2$}
\label{sec:gl3-gl2}
In this section, we study the $p$-adic $L$-functions that interpolate critical values of Rankin--Selberg $L$-functions
on $\GL_3 \times \GL_2$. For simplicity, we only study the cohomology groups with constant coefficients, and our exposition is based on Schimdt~\cite{schmidt-inv2}. The reader is referred to Kazhdan, Mazur and Schmidt~\cite{kms}, as well as recent papers of Januszewski (\cite{janus-crelle} \cite{janus-imrn} \cite{janus-preprint}) for generalization to 
$\GL_n \times \GL_{n-1}$ over a general number field and for representations having cohomology for more general coefficient systems.

\medskip
\subsection{\bf  $L$-functions for \texorpdfstring{$\GL_3 \times \GL_2$}{X}}
\label{Schmidt}
Let $(\pi, V_{\pi})$ be a cohomological, cuspidal, automorphic representation of $\GL_3(\A)$ 
and $(\sigma, V_{\sigma})$ be a cohomological,  cuspidal,  automorphic representation of $\GL_2(\A)$
with the decompositions $\pi=\otimes'_p \pi_p$ and $\sigma=\otimes'_p\sigma_p$ as restricted tensor 
products. The global Rankin\--Selberg $L$-function attached to $\pi$ and $\sigma$ is defined as an 
Euler product $L(s,\pi, \sigma)=\prod_{p \leq \infty} L(s, \pi_p,\sigma_p).$ By the work of Jacquet, Piatetskii-Shapiro and Shalika, we have an integral representation for this $L$-function (see the lecture notes by Cogdell~\cite{Cogdell}), 
which is exploited to construct $p$-adic measures.

\subsubsection{\bf Local $L$-functions}
Let $N_2$ be the set of unipotent matrices inside $\GL_2$ and 
consider the embedding  $j: \GL_2 \rightarrow \GL_3$ given by 
$j(g)=(\begin{smallmatrix}  g & 0  \\ 0 & 1 
 \end{smallmatrix})$. 
For any local Whitakker functions $w_p \in W(\pi_p,\psi_p)$ and $v_p \in W(\sigma_p, \psi_p^{-1})$, define  
\[
\Psi(s, w_p, v_p)=\int_{N_2(\Q_p) \backslash \GL_2(\Q_p)}w_p(j(g))v_p(g) |det(g)|_p^{s-\frac{1}{2}} dg.
\]
Such an integral converges for $\Re(s) \gg 0$ and has a meromorphic continuation to all of $\C$ 
as a rational function in $p^{-s}$. These integrals span a nonzero fractional ideal in $\C(p^s)$ with respect to the
subring $\C[p^s, p^{-s}]$. This ideal has a unique generator $P_p(p^{-s}),$
for a polynomial $P_p(X) \in \C[X]$ normalized so that $P_p(0)=1.$ 
The local $L$-function is defined as 
$L(s,\pi_p,\sigma_p)=P_p(p^{-s})^{-1}.$
At the infinite places, we can write  
$L(s, \pi_{\infty},\sigma_{\infty})$ as product of $\Gamma$ functions.

\subsubsection{\bf Local and global zeta integrals}
Let $\phi \in V_{\pi}$ be a cusp form on $\GL_3(\A)$ 
with $W_{\phi} \in W(\pi, \psi)$ the corresponding Whittaker function, and similarly,  
$\phi' \in V_{\sigma}$  on $\GL_2(\A)$ with $W_{\phi'} \in W(\sigma, \psi^{-1})$. Define a global period integral 
associated to $(\phi,\phi')$ as
 \[
 I(s, \phi, \phi')=\int_{\GL_2(\Q) \backslash \GL_2(\A)} \phi(j(h)) \, \phi'(h) \, 
 |\mathrm{det}(h)|^{s-\frac{1}{2}}\, dh. 
 \]
 After a standard unfolding argument ~\cite{Cogdell}, we have 
\begin{equation}
I(s, \phi, \phi') \ = \ 
\int_{N_2(\A) \backslash \GL_2(\A)} W_{\phi} (j(h))\,W_{ \phi'}(h) \, |det(h)|^{s-\frac{1}{2}}\, dh
\ =: \ \Psi(s, W_{\phi}, W_{\phi'}). 
\end{equation} 
Assume $\phi$ is a pure tensor so that $W_{\phi}(g)=\prod_p W_{\phi_p}(g_p)$. Similarly, $W'_{\phi'}(g)=\prod_p W'_{\phi'_p}(g_p)$. The global integral factors as a product of local integrals: 
\[
\Psi(s, W_\phi, W_{\phi'}) \ = \ 
\prod_p \int_{N_2(\Q_p) \backslash \GL_2(\Q_p)} W_{\phi_p}(j(h_p) ) W'_{\phi'_p}(h_p) |det(h_p)|^{s-\frac{1}{2}} dh_p
\ = \ \prod_p \Psi(s, W_{\phi_p}, W'_{\phi'_p}). 
\]

\subsubsection{\bf Zeta integrals and $L$-functions}
If $\pi_p$ and $\sigma_p$ both are spherical,  choose $w^0_p \otimes v^0_p$ to be the ``essential vector"
\cite{raghuram-imrn}. By the above choice, we have 
$\Psi(s, w^0_p, v^0_p)=L(s, \pi_p, \sigma_p)$.
For other primes $p$, we find ``good tensors" $t_p \in W(\pi_p, \psi_p) \otimes W(\sigma_p, \psi^{-1}_p)$
such that $ \Psi(s, t_p)=L(s, \pi_p,  \sigma_p).$
In general, we have 
\begin{equation}
\label{Whittaker}
t \ = \ \otimes t_p \ = \ \sum_{\iota=1}^n w_{\iota} \otimes v_{\iota} 
\end{equation}
as a decomposition in the Whittaker model $W(\pi, \psi) \otimes W(\sigma, \psi^{-1})$ 
as sum of pure tensors. Let $\phi_i \in V_{\pi}$ correspond to $w_i,$ and $\varphi_i \in V_{\sigma}$ 
correspond to $v_i.$ 
These cusp forms appear in the global Birch Lemma below. 
Note that the integral $I(s, \phi,\phi')=\prod_p \Psi(s, W_{\phi_p}, W_{\phi'_p})$ 
depends on the pure tensor $w \otimes v$ and the global  $L$-function 
$L(s, \pi, \sigma)=\prod_l L(s, \pi_l, \sigma_l)$
lies in the image of this map. For any choice of $(w_{\infty}, v_{\infty})$, there is an entire function $\Omega(s)$ such that 
\begin{equation}
\label{Birchequation}
 \Omega(s) L(s, \pi, \sigma) \ = \ \Psi(s, w_{\infty}, v_{\infty}) \prod_{l \neq \infty} \Psi(s, t_l).
\end{equation}
Here $t_l$ is a linear combination of pure tensors of the form $w_l \otimes v_l$.  The polynomial 
$\Omega(s)$ depends on the choice of $(w_{\infty}, v_{\infty})$.

\medskip
\subsection{\bf Birch's Lemma}
\label{BirchLemma}

\subsubsection{\bf The classical Birch's Lemma.} 
Consider a classical elliptic modular form $f$ of weight two. 
Recall, the $L$-function attached to this modular form $f$ can be defined in terms of Mellin 
transform as 
\[
 \frac{\Gamma(s)} {(2 \pi)^s}L(s, f)=\int_0^{\infty} f(it) t^{s-1} dt. 
\]
In particular, $L(1,f)=(2 \pi)  \int_0^{\infty} f(it) dt$. For any $a, m \in \Q$ with $m >0$, define the period integrals by 
\[
\lambda(a,m,f)=2 \pi \int_0^{\infty} f(it-\frac{a}{m}) dt. 
\]
For a primitive character $\chi$ of conductor $m$, using the Gauss sum of $\chi$, we have the interpolation formula: 
\[
\chi(n)=\frac{1}{G(\overline{\chi})}\sum_{a \text{ mod } m} \overline{\chi}(a) e^{{\frac{2 \pi i a n}{m}}}.
\]
In particular, we get
\[
f_{\chi}(z)=\sum_{n \geq 1} \chi(n) a_n e^{2 \pi i n z } \ = \ 
\frac{1}{G(\overline{\chi})}(\sum_{a \text{ mod } m} \overline{\chi}(a) f(z+\frac{a}{m})).
\]
By rearranging, we get the classical Birch's Lemma: 
\[
 L(1, f, \chi)=\frac{1}{G(\overline{\chi})} \left(\sum_{a \text{ mod } m} \overline{\chi}(a) \lambda(a, m, f)\right), 
\]
i.e., the value of  the $L$-function at the critical point $1$ can be written as linear combinations of periods.

\subsubsection{\bf Birch's Lemma for $\GL_3 \times \GL_2$.} 
Let $U_p =\Z_p^{\times}$ if $p \neq \infty$ and if $p = \infty$ let $U_\infty = \R_{+}^{\times}.$ 
We have a determinant map 
$\mathrm{det} : \GL_2(\Q) \backslash \GL_2(\A) \rightarrow \Q^{\times} \backslash \A^{\times}$.
For any $\alpha \in \A^{\times}$, define 
\[
C_{\alpha,f} 
\ := \ 
{\rm det}^{-1}\left(\Q^{\times}  \backslash \Q^{\times} \cdot (\alpha (1+f) \prod_{q \neq p} U_q) \right) 
\ \subset  \ \GL_2(\Q) \backslash \GL_2(\A). 
\]
Note that:
\[
 \GL_2(\Q) \backslash \GL_2(\A)=\bigcup_{\epsilon \, {\rm mod} \,f} C_{1,f} (\begin{smallmatrix} \epsilon  & 0  \\ 0 & 1 
 \end{smallmatrix})
\]
We now state the general global Birch's Lemma for $L$-functions on $\GL_3 \times \GL_2$ (see \cite{janus-crelle}). 
Let $t=\mathrm{diag}(f,1,1)$ and $h^{(f)}=t^{-1}h t$ as matrices in $\GL_3(\A)$.

\begin{thm}[Birch's Lemma]
\label{Birchlemma}
Let $\chi$ be a quasi-character on $\Z_p^{\times}$ of conductor $f=p^n$ 
and consider the pure tensors as in (\ref{Whittaker})
For any choice of $(w_{\infty}, v_{\infty})$  and for any Iwahori invariant pair $(w_p,v_p)$, the corresponding entire function $\Omega$ 
satisfies the following property 
\[
\Omega(s) \, \kappa(w_p,v_p,\chi,f) \,L(s, \pi \otimes \chi, \sigma) \ = \ 
\sum_{\iota} \int_{\GL_3(\Q) \backslash \GL_3(\A)} \phi_{\iota}(j(g) h^{(f)}) \, \varphi_{\iota}(g) \,\chi(\mathrm{det}(g))\, 
||{\rm det}(g)||^{s-\frac{1}{2}}dg.
\]
Here, $\kappa(w_p,v_p,\chi,f)=w_p(1_3) v_p(1_2) \prod_{v=1}^3 (1-p^{-v})^{-1} G(\chi)^6 \eta(f)$.
\end{thm}

\medskip
\subsection{$p$-adic measures and $p$-adic $L$-functions for $\GL_3 \times \GL_2$}

Given a pair $(\phi,\varphi)$ of Hecke eigenforms for $\GL_3/\Q$ and $\GL_2/\Q$, there exists a $\C$-valued 
distribution on $\Z_p^{\times}$ such that the special values of the Rankin--Selberg $L$-function 
$L(\frac{1}{2}, \pi \otimes \chi, \sigma)$ can be written as a $p$-adic integral of $\chi$ 
against this distribution. 
We  define a $p$-adic measure associated to cusp forms $(\phi, \varphi)$. 
Fix the following data: 
\begin{itemize}
\item a pair of roots $\lambda, \mu$ of the Hecke polynomial of $\phi$ at $p$ 
\item a root $\alpha$ of the Hecke polynomial of $\varphi$ at $p$.
\end{itemize}
For $g \in \GL_2(\A)$, let $g_p$ denote the $p$ component 
of $g$.  Define certain `partial periods' for $\phi, \varphi$ on $\GL_3(\A)$ and $\GL_2(\A)$ by 
\[
P(i,j,y,f) \ := \ 
P_{\lambda,\mu}^{\alpha}(i,j,y,f) \ = \ 
\int_{C_{1,f}} \sum_{\beta \, mod \, f} \phi_{\lambda,\mu}  (j(g)\left(\begin{smallmatrix}
1 & \frac{i}{f} & \frac{y+\beta f}{f^2} \\
 & 1 & \frac{j}{f} \\
 &   &  1\\
 \end{smallmatrix}\right)_p)\varphi_{\alpha}(g) dg
\]
for any $p$-power $f >1$ and $i,j$ mod $f$. 
The following distribution relations for these periods is due to Schmidt~\cite[Prop.\,4.4]{schmidt-inv2}. 
\begin{prop}
\label{Heckeoperatorsrelation}
The periods $P(i,j, y,f)$ satisfy the following distribution relation 
\[
\sum_{a,b,c=0}^{p-1}P(i+af,j+af, y+cf,fp) \ = \ \lambda^2 \mu \alpha \eta(p) p^{-3} P(i,j,y,f), 
\]
 for $\eta(p)$ as in~\cite[p.\,57]{schmidt-inv2}.
\end{prop}

For $m \geq 1$ and $i \in (\Z/p^m \Z)^{\times}$, 
define a function $\mu_{\pi, \sigma}$, on certain open subsets of $\Z_p^\times$ 
by 
\[
\mu_{\pi, \sigma}(i+p^n \Z_p) \ = \ 
\kappa^{-m} \sum_{y \, {\rm mod}\, p^m} P(i,1,y,p^m).
\]
Note that $\mu_{\pi, \sigma}$ depends on the choice of $\lambda, \mu$ and $\alpha.$ 
We call a representation $\pi$ on $\GL_3$ (resp., $\sigma$ on $\GL_2$) to be $p$-ordinary if the roots of the Hecke polynomial satisfy  $v_p(\lambda)=0$ and $v_p(\mu)=\frac{1}{p}$ (resp., $v_p(\alpha)=0$). 

\begin{thm}
Let $\pi$ and  $\sigma$ be  two $p$-ordinary representations  on $\GL_3$ and $\GL_2$. 
For  a suitable choice of Whittaker 
functions $(\tilde{w}_p, \tilde{v}_p)$, and for any $p$-adic character $\chi$ 
of finite order with non-trivial conductor $f=p^n$, we have 
\[
\int_{\Z_p^{\times}} \chi \, d \mu_{\pi, \sigma} \ = \ 
\Omega(\frac{1}{2}) \, \delta(\pi, \sigma) \, G(\chi)^3 \, \widehat{k}(f)
L(\frac{1}{2}, \pi \otimes \chi, \sigma),
\]
where  $\delta(\pi, \sigma) = w_p(1_3) v_p(1_2) \prod_{v=1}^3 (1-p^{-v})^{-1}$ 
 and $\widehat{k}(f)=(p^{-1}\alpha \lambda^2)^{-v_p(f)} $.
\end{thm}

For the proof, we refer the reader to Schmidt~\cite{schmidt-inv2} and Januszewski~\cite{janus-crelle}. 
See also, \cite[Thm.\,5.1]{janus-crelle}, where, under the $p$-ordinarity assumption, it is proved that 
$\mu_{\pi,\sigma}$ -- after renormalizaing by certain archimedean periods -- takes values in the ring of integers of a number field, and hence is a $p$-adic measure.

\medskip
\medskip
\subsubsection{\bf An interlude  on exceptional zeros and non-vanishing of $L$-function}
\label{exceptional}
The extra zeros of $p$-adic $L$-functions are zeros different from zeros of complex $L$-functions. 
 The precise information about these zeros will be required to study the $p$-adic $L$-functions 
for $\Sym^3 (\pi)$.  Greenberg-Stevens~\cite{Greenberg-Stevens} first proved ``exceptional zero conjecture" about these special zeros of the $p$-adic $L$-functions attached to modular forms (Sect.\,\ref{Modularforms}). 
We illustrate the phenomenon of extra zeros for the $p$-adic $L$-function attached to a modular form 
$f$ of weight 
two.  Recall that we have two types of $L$-functions attached to $f$: 
\begin{itemize}
 \item
 a complex L-function $L(s, f)$, which is a function in the complex variable $s$, and 
 \item 
a $p$-adic $L$-function $L_{p,f,\alpha} : X_p \rightarrow \C_p.$ 
\end{itemize}
These two functions are linked by the interpolation property (Thm.\,\ref{Mazur}) which we recall:

\begin{prop}
For a Dirichlet character $\psi$  of conductor $m=p^v M$ and Gauss sum $G(\psi)$, we have the following
interpolation property:
\[
L_{p,f, \alpha}(\psi)=e_{p, f,\alpha}(\psi) \frac{m}{G(\overline{\psi})} L(1,f \otimes \overline{\psi}).
\]
\end{prop}
The {\it Euler factor} $e_{p, f,  \alpha}$ at $p$ is given by:
\[
e_{p, f, \alpha}(\psi)=\frac{1}{\alpha^v}(1-\frac{\overline{\psi(p)} \epsilon(p)}{\alpha})(1-\frac{\psi(p)}{\alpha}).
\]
These Euler factors contribute to the ``extra zeros" of $L_{p,f, \alpha}$. For a Dirichlet character $\psi$,
$L_{p,f, \alpha}=0$ even if $L(1,f \otimes \overline{\psi}) \neq 0$ but $e_{p, f, \alpha}=0$. Hence,  zeros of $e_{p, f, \alpha}$ are zeros of 
$L_{p, f,\alpha}$ different from critical zeros of complex $L$-functions.

Recall the corresponding interpolation result for the $p$-adic $L$-function attached 
to an automorphic representation $\pi$ of $G=\GL_2/\Q$ with trivial coefficient systems. 
Say $p$ be an unramified prime for $\pi$, then the local component $\pi_p$ is a spherical representation of the form 
$\Ind_B^G(\mu,\mu^{-1})$. 
Let $\chi$ be a finite order id\'ele class character with conductor $c(\chi)$, $p$ component $\chi_p$ and the adelic Gauss sum $\tau(\chi)$.

\begin{thm} [\cite{spiess}, \S 4.6]
\label{speiss}
If $\alpha=\mu(p) \sqrt{p} \in O_p^{\times}$,  there exist a measure $\mu_{\pi} \in \ {\rm Hom}_{\C_p}(C^\infty(\Z_p^{\times}, \C_p), \, \C_p)$
with the following interpolation property:
\[
\int_{\Z^{\times}_p} \chi_p d \mu_{\pi}= \tau(\chi) e_{p, \pi, \alpha}(\chi_p) L(\frac{1}{2}, \pi \otimes \chi).
\]
The Euler factor $e_{p, \pi, \alpha}(\chi_p)$ at $p$ is given by
\begin{equation*}
e_{p, \pi, \alpha}(\chi_p) =
\begin{cases}
(1-\alpha \chi(p)^{-1}) & \text{if $v_p(c(\chi))=0$ and $\alpha= \pm 1$,} \\
(1- \frac{\chi(p)}{\alpha}) (1-\frac{1}{\alpha \chi(p)}) & \text{if $v_p(c(\chi))=0$ and  $\alpha \neq \pm 1$,} \\
{\alpha}^{-v_p(c(\chi))} & \text{  if } v_p(c(\chi))>0.
\end{cases}
\end{equation*}
\end{thm}
\smallskip
The following theorem of Rohrlich \cite{Rohrlich} is useful to define a  function $'L_{p, \Sym^3(\pi)}$ on a subset of $X_p$ that interpolate the critical values of $L$-functions for $\Sym^3(\pi)$. 

\begin{thm}
\label{Rohrlich}
 Let $\pi$ be an irreducible cuspidal automorphic representation of $\GL(2)/\Q$ and
 $S$ a finite set of primes.
There exist infinitely many primitive ray class characters $\chi$ such that $\chi$ is  
unramified on $S$ and $L(\frac{1}{2}, \pi \otimes \chi) \neq 0$.  
\end{thm}

\medskip
\subsection{$p$-adic symmetric cube $L$-function -- II}
\label{sec:sym-3-II}

Let $\pi$ be a cohomological, cuspidal automorphic representation of $\GL_2/\Q$. 
One may attempt to construct a $p$-adic $L$-function for $\mathrm{Sym}^3$ transfer of $\pi$ using the $p$-adic $L$-function attached 
to $\Sym^2(\pi) \times \pi$ (Sect.\,\ref{Schmidt}) and the $p$-adic 
$L$-function for $\pi$ (Sect.\,\ref{automorphic}).   
At the level of the complex $L$-functions, we have
$$
L(s, \mathrm{Sym}^2(\pi) \times \pi) \ = \ L(s, \mathrm{Sym}^3(\pi)) \, L(s, \pi \otimes \omega_\pi).
$$
It is natural 
 to define  a map $L_{p, \Sym^3(\pi)}:X_p \rightarrow \C_p$ as a quotient of the $p$-adic $L$-function for $\Sym^2(\pi) \times \pi$ and the $p$-adic $L$-function 
of $\pi \otimes \omega_\pi$. For simplicity of exposition, assume henceforth that $\omega_\pi$ is trivial. 

Let $Z_p$ be the subset of $X_p$ consisting of all $p$-adic characters $\chi$ such that 
$L_{p, \pi}(\chi)=0$. Define a function $'L_{p,\Sym^3(\pi)}:X_p-Z_p \rightarrow \C_p$ as 
\[
 'L_{p, \Sym^3(\pi)} (\chi) \ = \  \frac{L_{p,  \Sym^2(\pi) \times \pi}(\chi)}{L_{p, \pi}(\chi) }.
\]
\begin{lemma}
The set $X_p-Z_p$ is non-empty and the set $Z_p$ is finite. 
\end{lemma}

\begin{proof}
Using Thm.\,\ref{Rohrlich} with $S=\{p\}$, there are infinitely many id\`ele class characters $\chi$, unramified at $p$, and 
 for which the corresponding critical value
$$L(\frac{1}{2}, \pi \otimes \chi)\neq 0.$$
Since $|\mu(p)|_{\C}=1$ (Deligne's proof of Ramanujan's conjecture) and $\chi$ is a finite order character, $\chi(p) \neq \alpha^{\pm 1}$.
 For a character $\chi$ with $v_p(c(\chi))=0$
and $L(\frac{1}{2}, \pi \otimes \chi)\neq 0$, we get $L_{p,\pi}(\chi) \neq 0$ and hence the set $X_p-Z_p$ is non-empty.

Recall, the mapping $u \rightarrow  \psi \chi_u$ identifies the open unit disc  $\B$ of the Tate field  with the set of characters on $\Z_p^{\times}$ with tame part equal to $\psi$ (Lem.\,\ref{opendisc}).  
For a fixed $\psi$, consider the function $L_{p, \pi}$ on $\B$.
The $p$-adic $L$-function $L_{p, \pi}$ is a non-zero power series with coefficients in $O_p$
on $\B$. 
By the Weierstrass preparation theorem (Lem.\,\ref{finitezeros}), 
there are only finitely many zeros of this power series. For the Dirichlet character $\psi$, let $Z_{\pi,\psi}$ be the finite set of zeros of $L_{p,\pi}$ on $\B$. Since 
$Z_p=\cup Z_{\pi, \psi}$,
the set $Z_p$ is also finite. 
\end{proof}

The function $'L_{p,\Sym^3(\pi)}$ interpolates the critical values of $\Sym^3(\pi)$ on $X_p-Z_p$ 
and it is an element in the quotient field of the Iwasawa algebra $O_p[[T]]$
on  $X_p-Z_p$. We expect that $'L_{p, \Sym^3(\pi)}$ should be an element 
of the Iwasawa algebra $O_p[[T]]$ on  $X_p$ and it should be obtained as a Mellin transform of a $p$-adic measure.

It is an interesting problem to see if one can refine the intervening periods so that $'L_{p, \Sym^3(\pi)}$ coincides with
the $p$-adic $L$-function $L_{p, \Sym^3(\pi)}$ constructed in Sect.\,\ref{symcube1}.

\end{document}